\documentclass[10pt,oneside,shortlabels, reqno]{amsart}
\usepackage{color}
\usepackage{enumitem}
\usepackage{amsbsy}
\usepackage{amstext}
\usepackage{amsthm}
\usepackage{amssymb}
\usepackage[unicode=true,pdfusetitle,
 bookmarks=true,bookmarksnumbered=false,bookmarksopen=false,
 breaklinks=false,pdfborder={0 0 0},pdfborderstyle={},backref=page,colorlinks=true]
 {hyperref}
\hypersetup{
 linkcolor=blue}

\makeatletter

\usepackage{mathrsfs}

\usepackage[all]{xy}
\newcommand{\xyL}[1]{%
\xydef@\xymatrixrowsep@{#1}
} 
\newcommand{\xyC}[1]{%
\xydef@\xymatrixcolsep@{#1}
} 

\theoremstyle{plain}
\newtheorem{thm}{\protect\theoremname}[section]
\theoremstyle{remark}
\newtheorem{rem}[thm]{\protect\remarkname}
\newcommand\thmsname{Theorem}
 \newcommand\nm@thmtype{thm}
 \theoremstyle{plain}
 
 \newenvironment{namedthm}[1]{
   \renewcommand\thmsname{#1}\renewcommand\nm@thmtype{namedtheorem}
   \begin{\nm@thmtype}
}
   {\end{\nm@thmtype}
}
\theoremstyle{remark}
\newtheorem{remarks}[thm]{Remarks}
\theoremstyle{definition}
\newtheorem{example}[thm]{\protect\examplename}
\theoremstyle{definition}
\newtheorem{defn}[thm]{\protect\definitionname}
\theoremstyle{plain}
\newtheorem{prop}[thm]{\protect\propositionname}
\theoremstyle{plain}
\newtheorem{lem}[thm]{\protect\lemmaname}
\theoremstyle{definition}
\newtheorem{void}[thm]{}

\numberwithin{equation}{section}

\makeatother

\providecommand{\definitionname}{Definition}
\providecommand{\examplename}{Example}
\providecommand{\lemmaname}{Lemma}
\providecommand{\propositionname}{Proposition}
\providecommand{\remarkname}{Remark}
\providecommand{\theoremname}{Theorem}

\begin{document}
\global\long\def\acc#1#2{\accentset#1#2}%

\title{NORMAL FORMS OF HYPERBOLIC LOGARITHMIC TRANSSERIES}
\author{D. PERAN$^1$, M. RESMAN$^2$, J.P. ROLIN$^3$, T. SERVI$^4$}
\thanks{This research of D. Peran and M. Resman is partially supported by the Croatian Science Foundation (HRZZ) grant UIP-2017-05-1020. The research of M. Resman is also partially supported by the Croatian Science Foundation (HRZZ) grant PZS-2019-02-3055 from Research Cooperability funded by the European Social Fund. The research of all four authors is partially supported by the Hubert-Curien 'Cogito' grant 2021/2022 \emph{Fractal and transserial approach to differential equations}.}

\begin{abstract}
We find the normal forms of hyperbolic logarithmic transseries with
respect to parabolic logarithmic normalizing changes of variables. We provide a necessary
and sufficient condition on such transseries for the normal form to
be linear. The normalizing transformations are obtained via fixed point theorems,
and are given algorithmically, as limits of Picard sequences in appropriate
topologies.
\end{abstract}

\maketitle
\tableofcontents{}

\section{Introduction}

Given the germ $f$ of a real transformation at a fixed point, finding
a normal form for $f$ is a fundamental tool in the study of its dynamics.
The term \emph{normal form} usually means a simple representative
of the conjugacy class of $f$ with respect to a given type of conjugation\,:
formal, $\mathcal{C}^{\infty}$, holomorphic,\ldots{} There is a rising
interest in studying the normal form problem for one variable transformations
which admit a transserial asymptotic expansion, where by a \emph{transseries}
we mean a series whose monomials involve not only real powers, but
also iterated logarithms and exponentials of the variable (and the
coefficients are real numbers). This is motivated by the fact that
such transformations appear naturally in several problems in mathematics
and physics (see for example \cite{aniceto_ince_bacsar_schiappa:primer_resurgent_transseries_asymptotics}).
It is the case for example in the solutions of Dulac's problem on
non-accumulation of limit cycles on hyperbolic polycycles, in particular
in the works of Ecalle \cite{ecalle:dulac} and Ilyashenko \cite{ilyashenko:dulac},
where such series are investigated in detail. 

Unlike the normalization of classical formal power series in one variable,
which is completely understood \cite{carleson-gamelin:complex-dynamics,ilyashenko-yakovenko:lectures_analytic_differential_equations,Loray_series_divergentes},
the normalization of (formal) transseries turns out to be a more complicated
question. It is therefore important to solve it before tackling the
problem of normal forms for germs with a transserial asymptotic expansion.\\

In this spirit, our Main Theorem gives a complete answer for \emph{hyperbolic
logarithmic transseries}. For short, a transseries is called \emph{logarithmic}
if it does not involve the exponential function, and it is called
\emph{hyperbolic} if it can be written as $\lambda z+\cdots$, where
$\lambda>0$, $\lambda\ne1$ and $z$ is an infinitesimal
variable. The reader familiar with the theory of classification of
formal power series, holomorphic germs or $\mathcal{C}^{\infty}$
germs will be aware of the fact that, under the hyperbolicity hypothesis,
a germ can be \emph{linearized}, \emph{i.e. }it can be (formally,
analytically or smoothly) conjugated to its ``linear part'' $\lambda z$
via a \emph{parabolic} - \emph{i.e.} tangent to identity - change
of coordinates in the same regularity class. We show here that this
is not always the case for a hyperbolic logarithmic transseries. More
precisely, we completely describe the normal forms of such series,
and in particular we explain which ones can be actually linearized.
We may assume,without loss of generality, that $0<\lambda<1$. The normal forms for hyperbolic transseries for which $\lambda>1$ are then obtained by applying our results to inverse transseries. 
\\

This work continues the investigation started in \cite{mrrz1}. There,
the normalization problem is completely solved for so-called ``power-log''
transseries (namely, logarithmic transseries which do not involve
\emph{iterated-logarithmic} terms), whose first term has the form
$\lambda z^{\alpha}$, for some $\lambda>0$ and $\alpha>0$.
Here, we focus on \emph{hyperbolic} transseries, but we allow iterated-logarithmic
terms. The introduction of these new terms significantly complicates
the situation. It is important to study such transseries as well,
as they appear naturally in several problems, such as the solution
of Dulac's problem. The normalization of other types of logarithmic
transseries, in particular the parabolic ones ($\lambda=\alpha=1$), and strongly hyperbolic ones ($\lambda>0, \alpha \neq 1, \alpha >0$) will be studied in forthcoming papers.\\

The results of this paper improve those of \cite{mrrz1} in two directions.
First, as already said, we consider here iterated-logarithmic transseries.
Secondly, and possibly more importantly, we improve the proof methods
of \cite{mrrz1}. Indeed, formal normal forms were obtained there
via a transfinite term-by-term elimination process. Beside being convoluted
and highly non constructive (for example for computing the coefficients
of the monomials), this method has the drawback of making the proof
of the analogous statement for germs quite cumbersome. In particular,
the realization on the germ side of limit ordinal steps requires the
introduction of an inductive summation process, as it is shown in
\cite{mrrz:fatou}.\emph{ }Hence, one of our goals here is to dispense,
as far as possible, with this former impractical transfinite process.
We design a new algorithm based on a version of the Banach fixed point
theorem, inspired by \emph{Krasnoselskii's fixed point theorem} \cite{xiang_goergiev:noncompact_krasnoselskii_fixed_point}.
To this end, we endow our spaces of transseries with a distance which
gives them a structure of complete metric spaces. Then, the normalizing
transseries is obtained as the limit of a Picard sequence of a suitable
operator. It is worth noticing that by proceeding in this manner,
the normalizing map is built as the limit of a sequence of transseries
instead of the limit of a \emph{transfinite sequence.}

This particular fixed point result turns out to have many different
applications. Once the spaces of transseries are equipped with an
appropriate distance for which they are complete metric spaces, we
use it repeatedly to solve various functional or differential equations.
Hence we would like to stress its usefulness, even in problems which
are not directly related to normalization. 

We also introduce a generalized version of the Koenigs sequence associated
to a hyperbolic logarithmic transseries, and show that this sequence
converges to the normalizing transseries (even in the non-linearizable
case, in which the standard Koenigs sequence does not converge).\\

We will show in an upcoming paper how our new approach helps in solving
the normalization problem for some types of hyperbolic germs with
certain asymptotic expansions on complex domains, in particular germs
arising in Dulac's problem in the case of nondegenerate polycyles.\\

Before introducing all the necessary notation, let us add two comments
on the possible non-linearizability of hyperbolic transseries, which
marks an important difference with the case of classical hyperbolic
formal power series. First, this phenomenon has already been noticed
in \cite{mrrz1}, for the hyperbolic transseries considered there
(transseries without iterated logarithms). The normal form obtained
there is a special case of our result. 

Secondly, this phenomenon also appears in the framework of $\mathcal{C}^{1}$
maps. The pointwise or uniform convergence of Koenigs' sequences in
this context is studied in \cite{dewsnap-fischer:convergence-koenigs-sequences}.
The authors prove refinements of a result of O. Landford III which
establishes the uniform convergence, on a neigborhood in $0\in\mathbb{R}$,
of the Koenigs sequence of a map $f$ such that $f\left(0\right)=0$
and $f\left(x\right)=\lambda x+\mathcal{O}\left(x^{1+\varepsilon}\right)$
for some $\varepsilon>0$. They give in particular the example of
the map $f\left(x\right)=x\left(\lambda-\dfrac{1}{\log\left(x\right)}\right)$
in a positive neighborhood of the origin, and show that its Koenigs
sequence diverges on an interval $(0,a]$, for a small $a>0$. Their
main result in \cite[Th. 2.2]{dewsnap-fischer:convergence-koenigs-sequences}
gives a sufficient condition which guarantees the linearizability
and the convergence of the Koenigs sequence of a map $f\left(x\right)=\lambda x+\mathcal{O}\left(x\eta\left(x\right)\right)$,
where $0<\lambda<1$ and $\eta\left(x\right)$ is a product of powers
of iterated logarithms of the variable, such that $\eta\left(0\right)=0$.
It would be interesting in this respect, especially when the Koenigs
sequence diverges pointwise, to study the possible convergence, pointwise
or uniform, of our \emph{generalized} \emph{Koenigs sequences} \eqref{eq:koenigs}.

\section{The main result and its topological setting}

\subsection{Notation\label{subsec:Notation}}

We fix the following notation and describe the transseries we will
be working with.\\

Let $k\in\mathbb{N}$. We endow $\mathbb{R}\times\mathbb{Z}^{k}$
with the lexicographic order (denoted by $\leq$). For an infinitesimal
variable $z,$ we define $\boldsymbol{\ell}_{1}:=-\dfrac{1}{\ln\left(z\right)}$
and $\boldsymbol{\ell}_{k+1}:=\boldsymbol{\ell}_{k}\circ\boldsymbol{\ell}_{1}$
(for $k\in\mathbb{N}$). The \emph{monomials} $z^{\alpha}\boldsymbol{\ell}_{1}^{n_{1}}\cdots\boldsymbol{\ell}_{k}^{n_{k}}$,
where $\left(\alpha,n_{1},\ldots,n_{k}\right)\text{ are in }\mathbb{R}\times\mathbb{Z}^{k}$,
are ordered by the relation:
\[
z^{\alpha}\boldsymbol{\ell}_{1}^{n_{1}}\cdots\boldsymbol{\ell}_{k}^{n_{k}}\preceq z^{\beta}\boldsymbol{\ell}_{1}^{m_{1}}\cdots\boldsymbol{\ell}_{k}^{m_{k}}\text{ iff }\left(\beta,m_{1},\ldots,m_{k}\right)\leq\left(\alpha,n_{1},\ldots,n_{k}\right)\text{ in }\mathbb{R}\times\mathbb{Z}^{k}.
\]
In particular, $z^{\alpha}\preceq1$ for all $\alpha\in\mathbb{R}_{>0}$
(which corresponds to the fact that $z$ is infinitesimal) and $z\prec\boldsymbol{\ell}_{k}\prec\boldsymbol{\ell}_{m}$
iff $0<k<m$. 

For $k\in \mathbb{N}$, the class $\mathcal{L}_{k}$ of\emph{ logarithmic transseries
of depth $k$} is defined as the class of all transseries
$f$ of the form 
\begin{equation}
f=\sum_{\left(\alpha,n_{1},\ldots,n_{k}\right)\in\mathbb{R}\times\mathbb{Z}^{k}}a_{\alpha,n_{1},\ldots,n_{k}}\cdot z^{\alpha}\boldsymbol{\ell}_{1}^{n_{1}}\cdots\boldsymbol{\ell}_{k}^{n_{k}},\label{Transseries}
\end{equation}
where $\mathrm{Supp}\left(f\right):=\left\{ \left(\alpha,n_{1},\ldots,n_{k}\right)\in\mathbb{R}\times\mathbb{Z}^{k}:a_{\alpha,n_{1},\ldots,n_{k}}\ne0\right\} $,
called the \emph{support} of $f$, is a \emph{well-ordered} subset
of $\mathbb{R}_{\geq0}\times\mathbb{Z}^{k}$, and $a_{\alpha,n_{1},\ldots,n_{k}}\in\mathbb{R}$
for $(\alpha,n_{1},\ldots,n_{k})\in\mathrm{Supp}\left(f\right)$.
If in \eqref{Transseries} $\mathrm{Supp}\left(f\right)$ is a well-ordered
subset of $\mathbb{R}\times\mathbb{Z}^{k}$ instead of $\mathbb{R}_{\ge0}\times\mathbb{Z}^{k}$,
then we denote the space of such transseries by $\mathcal{L}_{k}^{\infty}$.
For $f\in\mathcal{L}_{k}$, $k\in\mathbb{N}$, as in \eqref{Transseries}
and $(\alpha,n_{1},\ldots,n_{k})\in\mathrm{Supp}\left(f\right)$,
we call 
\[
a_{\alpha,n_{1},\ldots,n_{k}}z^{\alpha}\boldsymbol{\ell}_{1}^{n_{1}}\cdots\boldsymbol{\ell}_{k}^{n_{k}}
\]
the\emph{ term of order }$(\alpha,n_{1},\ldots,n_{k})$ of\emph{ $f$}.
We set 
\[
[f]_{\alpha,n_{1},\ldots,n_{k}}:=a_{\alpha,n_{1},\ldots,n_{k}}.
\]
The \emph{order} $\mathrm{ord}\left(f\right)$ of $f$ is the minimal
element of $\mathrm{Supp}\left(f\right)$ (which exists because $\text{Supp}\left(f\right)$
is well-ordered). The \emph{leading term} of $f$, denoted by $\mathrm{Lt}(f)$,
and the \emph{leading monomial} of $f$, denoted by $\mathrm{Lm}(f)$,
are respectively the term and the monomial of order $\mathrm{ord}\left(f\right)$
of $f$.\\

For $0\leq m\leq k$, the set $\mathcal{L}_{m}$ will be identified
with a subset of $\mathcal{L}_{k}$ in the obvious way. The elements
of the increasing union $\mathfrak{L}:=\bigcup_{k\in\mathbb{N}}\mathcal{L}_{k}$
are the \emph{logarithmic transseries of finite depth} (or \emph{logarithmic
transseries}, for short). A transseries $f\in\mathfrak{L}$ is called
\emph{hyperbolic} if $\mathrm{Lt}\left(f\right)=\lambda z$, with
$\lambda>0$, $\lambda\neq1$, and \emph{parabolic} (or \emph{tangent to the
identity}) if $\mathrm{Lt}\left(f\right)=z$. The parabolic elements
of $\mathcal{L}_{k}$ form a group, denoted by $\mathcal{L}_{k}^{0}$,
with respect to the usual composition of transseries (see for example
\cite[Section 6]{vdd_macintyre_marker:log-exp_series}). Let $\mathfrak{L}^{0}:=\bigcup_{k\in\mathbb{N}}\mathcal{L}_{k}^{0}$.
We denote by $\mathcal{L}_{k}^{H}$ the collection of all elements
$f\in\mathcal{L}_{k}$ ``which do not contain logarithm in their
leading term'', \emph{i.e.} such that $\mathrm{Lm}\left(f\right)=z^{\alpha}$,
for some $\alpha>0$. When describing a transseries, we will often
use the acronym ``h.o.t.'' for ``higher order terms''. 
\begin{rem}
The collection $\mathfrak{L}$ introduced above is a subset of the
field of logarithmic-exponential series $\mathbb{R}\left(\left(t\right)\right)^{\mathrm{LE}}$
defined in \cite{vdd_macintyre_marker:log-exp_series}. Notice that
while the variable $t$ in $\mathbb{R}\left(\left(t\right)\right)^{\mathrm{LE}}$
is infinite, we prefer here to work with the infinitesimal variable
$z=t^{-1}$, which is more convenient in the framework of iteration
theory. Actually, $\mathfrak{L}$ is even contained in the subfield
$\mathbb{T}_{\mathrm{log}}$ of ``purely logarithmic transseries''
introduced in \cite{adh:towards_model_theory_transseries} and studied
from a model-theoretic point of view in \cite{gehret:nip_Tlog}. More
precisely, in $\mathfrak{L}$ the iterated logarithms are raised to
integer powers, whereas they are raised to arbitrary real powers in
$\mathbb{T}_{\mathrm{log}}$. 
\end{rem}

The following definitions and notation are related to the powers of
$z$ in a transseries. The monomials of our transseries can be organized
according to the powers of $z$, and this fact plays an important role
in the proof of our results. The \emph{support in $z$} of $f\in\mathfrak{L}$
is the set $\mathrm{Supp}_{z}(f)\subseteq\mathbb{R}_{\ge0}$ of all
exponents of $z$ in $f$. Likewise, the \emph{order of $f$ in $z$}
is the real number: 
\[
\mathrm{ord}_{z}(f):=\min\mathrm{Supp}_{z}(f).
\]
Notice that $f\in\mathcal{L}_{k}$ as in \eqref{Transseries} can
also be written \emph{blockwise}: 
\begin{equation}
f=\sum_{\alpha\in\mathrm{Supp}_{z}(f)}z^{\alpha}R_{\alpha},\label{eq:blockwise}
\end{equation}
where $R_{\alpha}\in\mathcal{L}_{k}$ is given by: 
\[
R_{\alpha}:=\sum_{(\alpha,n_{1},\ldots,n_{k})\in\mathrm{Supp}(f)}a_{\alpha,n_{1},\ldots,n_{k}}\boldsymbol{\ell}_{1}^{n_{1}}\cdots\boldsymbol{\ell}_{k}^{n_{k}}.
\]
For $\alpha\in\mathrm{Supp}_{z}(f)$, we call $z^{\alpha}R_{\alpha}$
the \emph{$\alpha$-block of $f$}. For $\alpha_{0}:=\mathrm{ord}_{z}(f)$,
the block $z^{\alpha_{0}}R_{\alpha_{0}}$ is called \emph{the leading
block in $z$} of $f$, and is denoted by $\mathrm{Lb}_{z}(f)$. In
the sequel, we use the acronym $\mathrm{h.o.b.}(z)$ for \textit{higher-order
blocks in $z$}.\\

Finally, for $\beta\ge0$ and $k\in\mathbb{N}$, we denote by $\mathcal{L}_{k}^{\geq\beta}\subseteq\mathcal{L}_{k}$
the set: 

\begin{equation}
\mathcal{L}_{k}^{\geq\beta}=\left\{ {f}\in\mathcal{L}_{k}:\mathrm{Supp}_{z}({f})\geq\beta\right\} .\label{eq:L greater equal than beta}
\end{equation}

\subsection{The main result}

In what follows, for $a\in\mathbb{R}$ and $p\in\mathbb{N}\setminus\left\{ 0\right\} $,
we will set $\boldsymbol{a}_{p}:=\left(a,\ldots,a\right)\in\mathbb{R}^{p}$.
\begin{namedthm}
{Main Theorem}Let $f=\lambda z+\mathrm{h.o.t.}\in\mathcal{L}_{k}$,
with $0<\lambda<1$, be a hyperbolic logarithmic transseries. Write
\[
f=\lambda z+\sum_{\boldsymbol{0}_{k}<\boldsymbol{m}\leq\boldsymbol{1}_{k}}a_{\boldsymbol{m}}z\boldsymbol{\ell}_{1}^{m_{1}}\cdots\boldsymbol{\ell}_{k}^{m_{k}}+\mathrm{h.o.t.},\quad\boldsymbol{m}=(m_{1},\ldots,m_{k}),
\]
and 
\begin{equation}
f_{0}=\lambda z+\sum_{\boldsymbol{0}_{k}<\boldsymbol{m}\leq\boldsymbol{1}_{k}}a_{\boldsymbol{m}}z\boldsymbol{\ell}_{1}^{m_{1}}\cdots\boldsymbol{\ell}_{k}^{m_{k}}.\label{eq:normal form}
\end{equation}
Then: 
\begin{enumerate}[1., font=\textup, nolistsep, leftmargin=0.6cm]
\item There exists a unique parabolic logarithmic transseries $\varphi=z+\mathrm{h.o.t.}\in\mathfrak{L}^{0}$,
called \emph{the normalizing transformation, }such that: 
\begin{equation}
\varphi\circ f\circ\varphi^{-1}=f_{0}.\label{eq:conjugacy equation theorem}
\end{equation}
Moreover, $\varphi\in\mathcal{L}_{k}^{0}.$
\item The normal form $f_{0}$ is \emph{minimal} with respect to inclusion
of supports, within the conjugacy class of $f$ by parabolic transformations
in $\mathfrak{L}^{0}$. In particular, $f$ can be linearized in $\mathfrak{L}^{0}$
if and only if $\mathrm{ord}(f-\lambda\cdot\mathrm{id})>\mathbf{1}_{k+1}$.
Moreover, the coefficients of $f_{0}$ are invariant under the conjugacy
class $\mathfrak{L}_{0}$. 
\item For a parabolic initial condition $h\in\mathfrak{L}^{0}$, the \emph{generalized
Koenigs sequence} 
\begin{eqnarray}
\Big(f_{0}^{\circ\left(-n\right)}\circ h\circ f^{\circ n}\Big)_{n}\label{eq:koenigs}
\end{eqnarray}
converges to $\varphi$ if and only if $\mathrm{Lb}_{z}(h)=\mathrm{Lb}_{z}(\varphi)$. 
\end{enumerate}
\end{namedthm}
Let us add a few remarks.
\begin{remarks}
\label{rem:post theorem}1. In part 3 of the statement of the Main
Theorem, the generalized Koenigs sequence converges in a certain
topology on $\mathfrak{L}$, which we will call the \emph{weak topology},
defined in Section \ref{subsec:two topologies}. It is interesting
to notice that this sequence may not converge in the more usual,
but finer, valuation topology on transseries (also introduced in Section
\ref{subsec:two topologies}). For example, consider $f=\lambda z+z^{2}$
with $0<\lambda<1$ and $h=\mathrm{id}$ (which are both classical power series).
In this case $f_{0}=\lambda z$, so the generalized Koenigs sequence
coincides with the standard Koenigs sequence and is given by $h_{n}=z+a_{n}z^{2}+\mathrm{h.o.t.}$,
where $a_{n}=\dfrac{1}{\lambda}+1+\lambda+\lambda^{2}+\cdots+\lambda^{n-2}$.
This coefficient never vanishes nor stabilizes, so the sequence $\left(h_{n}\right)$
does not converge in the valuation topology. However, $\left(a_{n}\right)$
tends to $\dfrac{1}{\lambda}+\dfrac{1}{1-\lambda}$ for the Euclidean
topology, in accordance with our main result.

2. Notice that, for a hyperbolic transseries $f\in\mathcal{L}_{k}$,
there exists a unique normalization $\varphi\in\mathfrak{L}^{0}$.
This means that not only there exists a unique normalization in $\mathcal{L}_{k}^{0}$,
but also that we cannot find another one in some $\mathcal{L}_{m}^{0}$,
with $m>k$. 

3. In some cases, as in the classical situation, the normalization
is a linearization. More precisely, according to the Main Theorem,
if $f=\lambda z+\mathrm{h.o.t.}\in\mathcal{L}_{k},\text{ with}\ 0<\lambda<1,$
is such that $\mathrm{ord}(f-\lambda\cdot\mathrm{id})>\mathbf{1}_{k+1}$,
then there exists a unique parabolic linearization $\varphi=z+\mathrm{h.o.t.}\in\mathfrak{L}^{0}$
such that
\[
\varphi\circ f\circ\varphi^{-1}=\lambda z.
\]
Moreover, $\varphi\in\mathcal{L}_{k}^{0}$. 

4. Whenever $\mathrm{ord}(f-\lambda\cdot\mathrm{id})>\mathbf{1}_{k+1}$,
for a parabolic initial condition $h\in\mathfrak{L}^{0}$, the generalized
Koenigs sequence coincides with the standard Koenigs sequence
\[
\Big(\frac{1}{\lambda^{n}}h\circ f^{\circ n}\Big)_{n}.
\]
It converges in the \textit{\emph{weak topology}} to $\varphi$, as
$n\to\infty$, if and only if $\mathrm{Lb}_{z}(h)=\mathrm{Lb}_{z}(\varphi)$.
So we obtain for hyperbolic logarithmic transseries the usual results:
if $f$ is linearizable, then its Koenigs sequence converges, in an
appropriate topology, to the linearizing map $\varphi$ (if its initial
condition is sufficiently close to $\varphi$).

5. The support of the normalization $\varphi$ depends only on the
support of the original transseries $f$, and not on the support of
the initial condition $h\in\mathfrak{L}^{0}$. More precisely, we
show in Section \ref{sec:support normalization} that the support
of $\varphi$ is contained in a set which is defined inductively and
depends solely on $\mathrm{Supp}\left(f\right)$.

6. The proof, given in Section~\ref{sec:proof theorem}, of the first
part of the Main Theorem is constructive:
it gives an algorithm for the construction of the normalization $\varphi$,
see Remark~\ref{rem:constr}. The leading block $\mathrm{Lb}_{z}(\varphi-\mathrm{id})$
is obtained as the limit of the Picard sequence of a suitable contraction
operator. Once we know $\mathrm{Lb}_{z}(\varphi-\mathrm{id})$, then $\mathrm{Lb}_{z}(\varphi )=z+\mathrm{Lb}_{z}(\varphi-\mathrm{id})$. Therefore, the\emph{ }(generalized) Koenigs
sequence applied to any initial condition $h\in\mathfrak{L}^{0}$
with the same leading block $\mathrm{Lb}(h)=\mathrm{Lb}(\varphi)$
converges to the normalization $\varphi$. 

7. Depending on the leading block of $f-\lambda z$, the support of
its normal form $f_{0}$ may not be finite, as it is illustrated by
the last of the following examples.
\end{remarks}

\begin{example}
Let $0<\lambda<1.$
\begin{enumerate}[1., font=\textup, nolistsep, leftmargin=0.6cm]
\item $f(z)=\lambda z+z\boldsymbol{\ell}_{1}+\mathrm{h.o.t.}\in\mathcal{L}_{1}$,
$f_{0}(z)=\lambda z+z\boldsymbol{\ell}_{1}$, $\varphi\in\mathcal{L}_{1}^{0}$, 
\item $f(z)=\lambda z+z\boldsymbol{\ell}_{1}^{2}+\mathrm{h.o.t.}\in\mathcal{L}_{1}$,
$f_{0}(z)=\lambda z$, $\varphi\in\mathcal{L}_{1}^{0}$, 
\item $f(z)=\lambda z+z(\boldsymbol{\ell}_{2}^{2}+\boldsymbol{\ell}_{1}\boldsymbol{\ell}_{2}^{-1}+\boldsymbol{\ell}_{1}\boldsymbol{\ell}_{3}+\boldsymbol{\ell}_{1}\boldsymbol{\ell}_{2}^{2}\boldsymbol{\ell}_{3}+\boldsymbol{\ell}_{1}^{2})+z^{2}\boldsymbol{\ell}_{1}^{-3}+z^{3}\in\mathcal{L}_{3}$,\\
 $f_{0}(z)=\lambda z+z(\boldsymbol{\ell}_{2}^{2}+\boldsymbol{\ell}_{1}\boldsymbol{\ell}_{2}^{-1}+\boldsymbol{\ell}_{1}\boldsymbol{\ell}_{3})$,
$\varphi\in\mathcal{L}_{3}^{0}$, 
\item $f(z)=\lambda z+z(\sum_{k\in\mathbb{N}}\boldsymbol{\ell}_{2}^{k}+\boldsymbol{\ell}_{1}\boldsymbol{\ell}_{2}+\boldsymbol{\ell}_{1}^{2})+\sum_{k\geq2}z^{k}\in\mathcal{L}_{2},$\\
 $f_{0}(z)=\lambda z+z(\sum_{k\in\mathbb{N}}\boldsymbol{\ell}_{2}^{k}+\boldsymbol{\ell}_{1}\boldsymbol{\ell}_{2})$,
$\varphi\in\mathcal{L}_{2}^{0}$. 
\end{enumerate}
\end{example}

\subsection{\label{subsec:two topologies}Two topologies on $\mathfrak{L}$.}

As the Main Theorem and several proofs refer to the notions of limit,
contraction and fixed point in a complete metric space, we introduce
the two main topologies which serve our purpose. 

\subsubsection{\label{subsec:valuation topology}The valuation topology}

The ordered valued field $\mathbb{R}\left(\left(t\right)\right)^{\mathrm{LE}}$
of transseries is classically endowed with the interval topology,
which coincides with the valuation topology (see \cite[p. 65 and 72]{vdd_macintyre_marker:log-exp_series}).
It turns out that the induced topology on the class $\mathcal{L}_{k}$,
for $k\in\mathbb{N}$, coincides with the metric topology defined
by the following distance $d_{z}:\mathcal{L}_{k}\times\mathcal{L}_{k}\to\mathbb{R}_{\geq0}$
: 
\begin{equation}
d_{z}(f,g)=\begin{cases}
2^{-\mathrm{ord}_{z}({f}-{g})}, & f\neq g,\\
0, & f=g,\ \quad f,\ g\in\mathcal{L}_{k}.
\end{cases}\label{eq:distance dz}
\end{equation}
Note that, for
every $k\in\mathbb{N}$, $(\mathcal{L}_{k},d_{z})$ is a complete
metric space by Proposition~\ref{prop:complete spaces}.

The valuation topology plays an important role in the proof of the
first part of the Main Theorem. Notice furthermore that, given a sequence
$\left(\varphi_{n}\right)_{n\in\mathbb{N}}$ of elements of $\mathcal{L}_{k}$,
the series $\sum_{n=0}^{\infty}\varphi_{n}$ converges in this topology
if and only if the sequence $\left(\mathrm{ord}_{z}\varphi_{n}\right)_{n\in\mathbb{N}}$
tends to $+\infty$. This remark is used in Section \ref{subsec:composition--transseries}
for the composition of transseries.

\subsubsection{\label{subsec:product topology}The weak topology}

This topology is used in the final part of the Main Theorem. Indeed,
we already noticed in Remark \ref{rem:post theorem} that the generalized
Koenigs sequence defined by \eqref{eq:koenigs} does not converge
in general for the valuation topology. Hence the need for a coarser
topology on our spaces of logarithmic transseries. To this end, we
consider $\mathcal{L}_{k}$ as a subset of the Cartesian product $\mathbb{R}^{\mathbb{R}_{\ge0}\times\mathbb{Z}^{k}}$
endowed with the product topology, where
each factor is itself equipped with the Euclidean topology. A sequence
$\left(\varphi_{n}\right)_{n\in\mathbb{N}}$ in $\mathcal{L}_{k}$
converges to $\varphi\in\mathcal{L}_{k}$ with respect to the weak topology if and only if
\[
\forall\left(\alpha,n_{1},\ldots,n_{k}\right)\in\mathbb{R}_{\ge0}\times\mathbb{Z}^{k},\lim_{n\rightarrow+\infty}\left[\varphi_{n}\right]_{\alpha,n_{1},\ldots,n_{k}}=\left[\varphi\right]_{\alpha,n_{1},\ldots,n_{k}}
\]
for the Euclidean topology.

This topology is related to the notion of summable families introduced
in \cite[p. 66]{vdd_macintyre_marker:log-exp_series}, in the following
way. If $G$ denotes the ordered group of logarithmic monomials, then
a family $\left(a_{i}\right)_{i\in I}$ of elements of the Hahn field
$\mathbb{R}\left(\left(G\right)\right)$ is called \emph{summable}
if the union $\bigcup_{i\in I}\mathrm{Supp}\left(a_{i}\right)$
is a well-ordered subset of $G$ and if, for each $g\in G$, there are only finitely many
$i\in I$ with $g\in\mathrm{Supp}\left(a_{i}\right)$. One can easily
check that, for $k\in\mathbb{N}$, if $\left(\varphi_{n}\right)$
is a sequence of elements of $\mathcal{L}_{k}$ such that the family
$\left(\varphi_{n}\right)_{n\in\mathbb{N}}$ is a summable family,
then the series $\sum_{n=0}^{\infty}\varphi_{n}$ converges in the
weak topology to the sum of the family. This last remark is also used
in Section \ref{subsec:composition--transseries}.\\

In Section \ref{subsec:gallery complete spaces}, we introduce another
metric in order to study the convergence of sequences within a given
block. 

\section{Preliminaries to the proof of the Main Theorem}

\label{sec:proof of the theorem}

The main idea of the proof of the Main Theorem consists in transforming
the \emph{conjugacy equation} \eqref{eq:conjugacy equation theorem}
into an equivalent \emph{fixed point equation}. Then, a suitable \emph{fixed
point theorem} gives us a unique solution. Thus we obtain the normalizing
change of variables as the limit of a Picard sequence in the valuation topology. Finally we
prove that, for well chosen initial conditions, the (generalized)
Koenigs sequence also converges to the normalizing change of variables, but in the weak topology.\\

\subsection{A fixed point theorem}

We will make a repeated use of the next proposition throughout the
proof of the Main Theorem.\\

Let us first recall a classical definition.
\begin{defn}[Homothety]
Let $\lambda,\mu>0$ and let $(X,d)$, $(Y,\rho)$ be two metric
spaces. 

A map $\mathcal{T}:X\to Y$ such that 
\[
\rho\left(\mathcal{T}(x),\mathcal{T}(y)\right)=\lambda d(x,y),\quad\forall x,y\in X,
\]
is called a \emph{$\lambda$-homothety}.

A map $\mathcal{S}:X\longrightarrow Y$ such that
\[
\rho\left(\mathcal{S}\left(x\right),\mathcal{S}\left(y\right)\right)\leq\mu d\left(x,y\right),\quad\forall x,y\in X,
\]
is called a \emph{$\mu$-Lipschitz map}.
\end{defn}

\noindent Notice that every $\lambda$-homothety $\mathcal{T}:X\to Y$
is injective on $X$. If, additionally, $\mathcal{T}$ is bijective,
then $\mathcal{T}^{-1}:Y\to X$ is a $\frac{1}{\lambda}$-homothety.\\

The following proposition, which is an easy consequence of the Banach
fixed point theorem, is inspired by a version of the fixed point theorem
due to Krasnoselskii (see e.g. \cite{xiang_goergiev:noncompact_krasnoselskii_fixed_point}).
\begin{prop}[A fixed point theorem]
\label{prop:krasno fixed point} Let $X$ be a complete metric space.
Let $\mathcal{S},\,\mathcal{T}:X\to X$, such that: 
\begin{enumerate}[1., font=\textup, nolistsep, leftmargin=0.6cm]
\item $\mathcal{S}$ is a $\mu$-Lipschitz map, 
\item $\mathcal{T}$ is a $\lambda$-homothety, 
\item $\mu<\lambda$, 
\item $\mathcal{S}(X)\subseteq\mathcal{T}(X)$. 
\end{enumerate}
Then there exists a unique point $x\in X$ such that $\mathcal{T}(x)=\mathcal{S}(x)$. 
\end{prop}

\begin{proof}
Since $\mathcal{S}(X)\subseteq\mathcal{T}(X)$, $\mathcal{T}^{-1}\circ\mathcal{S}:X\to X$
is well defined. The map $\mathcal{T}$ is a $\lambda$-homothety,
so its inverse $\mathcal{T}^{-1}$ is a $\frac{1}{\lambda}$-homothety
on $\mathcal{S}(X)$. Therefore, since $\frac{\mu}{\lambda}<1$, $\mathcal{T}^{-1}\circ\mathcal{S}:X\to X$
is a $\frac{\mu}{\lambda}$-contraction on $X$. We conclude by Banach's
fixed point theorem.
\end{proof}

\subsection{\label{subsec:composition--transseries}Composition of transseries}

In this work we often compose transseries. It turns out that a general
result on composition of transseries, the proof of which is rather
entangled, is given in \cite[Section 6]{vdd_macintyre_marker:log-exp_series}.
Fortunately, as we are only concerned with logarithmic transseries,
we need a much simpler version of this operation, which is mostly
based on Taylor's formula in our setting. However, in view of our
results, we need to interpret the convergence of the infinite sum
given by Taylor's formula in terms of the valuation and weak topologies
introduced in Section \ref{subsec:two topologies}. Since the convergence
of the Taylor series in either of the two topologies does not follow
directly from \cite{vdd_macintyre_marker:log-exp_series}, we prove
the following result.
\begin{prop}[Composition in the class $\mathcal{L}_{k}$, for $k\in\mathbb{N}$]
\emph{}\label{prop:formal Taylor} Let $f\in \mathcal{L}_{k}$ and $g\in\mathcal{L}_{k}^{H}$.
Let $g=\lambda z^{\alpha}+z^{\alpha}Q+g_{1}=g_{0}+g_{1}$, for $\alpha>0$,
where $g_{0}:=\lambda z^{\alpha}+z^{\alpha}Q$ is an initial part
of the blockwise representation \eqref{eq:blockwise} of $f$, and
the order $\mathrm{\mathrm{ord}}\left(g_{1}\right)$ is strictly bigger
than the orders of all the elements of $\mathrm{Supp}\left(g_{0}\right)$.
Then the composition ${f}\circ{g}\in\mathcal{L}_{k}$ is defined
as: 
\begin{eqnarray}
f\circ g:=f\left(g_{0}\right)+\sum_{i\geq1}\frac{f^{(i)}\left(g_{0}\right)}{i!}\left(g_{1}\right)^{i},\label{eq:comp taylor}
\end{eqnarray}
where the transseries on the right-hand side converges in the weak
topology on $\mathcal{L}_{k}$. If moreover $\mathrm{ord}_{z}\left(g_{1}\right)>\alpha,$
then the convergence holds also in the finer valuation topology on
$\mathcal{L}_{k}$. 
\end{prop}

\noindent Note that the right-hand side of \eqref{eq:comp taylor}
is just the formal Taylor expansion of ${f}(g_{0}+{g}_{1})$ at $g_{0}$.
This formula is particularly useful when $g_{0}=\lambda z^{\alpha}$.
\begin{proof}
It follows from the properties of the composition of transseries proved
in \cite{vdd_macintyre_marker:log-exp_series}, as a consequence of
Neumann's Lemma \cite{neumann:ordered_division_rings}, that the family 
\[
\left(\frac{f^{\left(i\right)}\left(g_{0}\right)}{i!}\left(g_{1}\right)^{i}\right)_{i\in\mathbb{N}}
\]
is summable in the sense recalled in Section \ref{subsec:product topology}.
Hence, as it was already noticed in the same section, the infinite
sum in \eqref{eq:comp taylor} converges in the weak topology.

For the second part of the statement, set $\mu:=\mathrm{ord}_{z}\left(f\right)$
and $r:=\mathrm{ord}_{z}\left(g_{1}\right)-\alpha$, $r>0$. It follows
from \eqref{eq:i-th derivative block} that 
\[
\mathrm{ord}_{z}\left(\frac{f^{\left(i\right)}\left(g_{0}\right)}{i!}\left(g_{1}\right)^{i}\right)=\left(\mu-i\right)\alpha+\left(\alpha+r\right)i=\mu\alpha+ri
\]
tends to $+\infty$ when $i\rightarrow+\infty$. We conclude by the
remark at the end of Subsection \ref{subsec:valuation topology}. 
\end{proof}

\subsection{The main strategy: the conjugacy equation via a fixed point equation}

Inspired by classical methods (\emph{e.g.} for $\mathcal{C}^{r}$
germs \cite[Ch. 3]{navas:groups_circle_diffeomorphisms}), our goal
in this section is to rewrite the \emph{conjugacy equation \eqref{eq:conjugacy equation theorem}}
as a \emph{fixed point equation}. \\

Consider a hyperbolic transseries: 
\begin{align}
f  & =\lambda z+\sum_{\boldsymbol{0}_{k}<\boldsymbol{m}\leq\boldsymbol{1}_{k}}a_{\boldsymbol{m}}z\boldsymbol{\ell}_{1}^{m_{1}}\cdots\boldsymbol{\ell}_{k}^{m_{k}}+\mathrm{h.o.t.}\nonumber\\
 & =f_{0}+g,\quad 0<\lambda<1,\ \boldsymbol{m}=(m_{1},\ldots,m_{k}),\label{eq:decomp f0 plus g}
\end{align}
where $f_{0}$ is the formal normal form given in the statement of
the Main Theorem, and $\mathrm{ord}(g)>\mathbf{1}_{k+1}$. Let
\[
g_{0}:=f_0-\lambda\cdot\mathrm{id}.
\]

\noindent Classically, the solution to the conjugacy equation \eqref{eq:conjugacy equation theorem}
is seen as a fixed point of the operator $\mathcal{P}_{f}:\mathfrak{L}^{0}\to\mathfrak{L}^{0}$:
\begin{equation*}
\mathcal{P}_{f}\left(\varphi\right):=f_{0}^{-1}\circ\varphi\circ f,\ \varphi\in\mathfrak{L}^{0},
\end{equation*}
where $f_{0}^{-1}$ denotes the compositional inverse of $f_{0}$.
We introduce the space 
\[
\mathfrak{L}_{>\mathrm{id}}:=\{h\in\mathcal{L}_{k}:\ \mathrm{ord}(h)>(1,0,\ldots,0)_{k+1},\ k\in\mathbb{N}\}\subseteq\mathfrak{L}.
\]
and write the solution as $\varphi=\mathrm{id}+h$, where $h\in\mathfrak{L}_{>\mathrm{id}}$.
Then $h$ is a fixed point of the operator $\mathcal{H}_{f}:\mathfrak{L}_{>\mathrm{id}}\to\mathfrak{L}_{>\mathrm{id}}$
defined by: 
\begin{align}
\mathcal{H}_{f}(h) & :=\mathcal{P}_{f}\big(\mathrm{id}+h\big)-\mathrm{id},\ h\in\mathfrak{L}_{>\mathrm{id}}. \nonumber 
\end{align}
Unfortunately, a simple computation shows that neither of the operators
$\mathcal{P}_{f}$ and $\mathcal{H}_{f}$ is a contraction for any
reasonable distance on $\mathfrak{L}^{0}$ or $\mathfrak{L}_{>\mathrm{id}}$.
Hence, in the next proposition, we reformulate the conjugacy equation
\eqref{eq:conjugacy equation theorem} as a fixed point problem in
the form of Proposition \ref{prop:krasno fixed point}. We omit the
proof, which is straightforward. 
\begin{prop}
\emph{\label{prop:fixed point equivalent conjugacy}} Let $k\in\mathbb{N}$
and $f(z)=\lambda z+\mathrm{h.o.t}\in\mathcal{L}_{k}$, with $0<\lambda<1$.
Let $f_{0}$ and $g$ be as in \eqref{eq:decomp f0 plus g}. For $\varphi\in\mathfrak{L}^{0}$
and $h:=\varphi-\mathrm{id}\in\mathfrak{L}_{>\mathrm{id}}$, the following
equations are equivalent: 
\begin{enumerate}[1., font=\textup, nolistsep, leftmargin=0.6cm]
\item $\varphi\circ f\circ\varphi^{-1}=f_{0}$, 
\item $\mathcal{T}_{f}(h)=\mathcal{S}_{f}(h),$ 
\end{enumerate}
\noindent where the operators $\mathcal{S}_{f},\mathcal{T}_{f}\colon\mathfrak{L}_{>\mathrm{id}}\to\mathfrak{L}_{>\mathrm{id}}$
are given by: 
\begin{align}
\mathcal{S}_{f}(h) & :=\frac{1}{\lambda}\bigg(g+\left(h\circ f-h\circ f_{0}\right)-\left(g_{0}\left(\mathrm{id}+h\right)-g_{0}-g_{0}'\cdot h\right)\bigg),\nonumber \\
\mathcal{T}_{f}(h) & :=\dfrac{1}{\lambda}\bigg(\left(\lambda\cdot h-h\left(\lambda\cdot\mathrm{id}\right)\right)-\left(h\circ f_{0}-h\left(\lambda\cdot\mathrm{id}\right)\right)+g_{0}'\cdot h\bigg),\label{eq:operators S and T}
\end{align}
where $g_{0}:=f_{0}-\lambda\cdot\mathrm{id}$. 
\end{prop}

The idea of the proof of the Main Theorem mainly consists in proving
that the operators $\mathcal{S}_{f}$ and $\mathcal{T}_{f}$ of Proposition
\ref{prop:fixed point equivalent conjugacy} are respectively a $\mu$-Lipschitz
map and a $\lambda$-homothety, where $\mu<\lambda$, on some complete
metric spaces, to which we can apply Proposition \ref{prop:krasno fixed point}.
As explained in Section \ref{sec:proof theorem}, we distinguish two
cases, depending on $\mathrm{ord}_{z}\left(g\right)$
in decomposition \eqref{eq:decomp f0 plus g}:
\begin{enumerate}[(a), font=\textup, nolistsep, leftmargin=0.6cm]
\item $\mathrm{ord}_{z}({g})>1$, 
\item $\mathrm{ord}_{z}({g})=1$. 
\end{enumerate}
The reason for this two-case approach is the following. Let 
\begin{align}
\mathcal{L}_{f} &:=\mathcal{S}_{f}-\frac{1}{\lambda}g \nonumber 
\end{align}
be the non-constant part of the operator $\mathcal{S}_{{f}}$ from
\eqref{eq:operators S and T}.\\

In case\emph{ }(a), using Taylor's formula (Proposition \ref{prop:formal Taylor}) and \eqref{eq:i-th derivative block}, we observe that, for ${h}\in\big\{{h}\in\mathcal{L}_{k}:\ \mathrm{ord}_{z}({h})\geq\mathrm{ord}_{z}({g})\big\}$,
$\mathcal{L}_{{f}}$ strictly increases the order of ${h}$ in $z$
(by the constant value $\mathrm{ord}_{z}(g)-1$, which depends only
on ${f}$): 
\begin{equation}
\mathrm{ord}_{z}(\mathcal{L}_{f}(h))\geq\mathrm{ord}_{z}(h)+\mathrm{ord}_{z}(g)-1.\label{eq:ell}
\end{equation}
It follows that the operator $\mathcal{S}_{{f}}$ in \eqref{eq:operators S and T}
is a contraction for the distance $d_{z}$ on the subspace $\big\{{h}\in\mathcal{L}_{k}:\ \mathrm{ord}_{z}({h})\geq\mathrm{ord}_{z}({g})\big\}$
of $\mathfrak{L}_{>\mathrm{id}}$. We can hence apply Proposition~\ref{prop:krasno fixed point}
to the equation $\mathcal{T}_{f}\left(h\right)=\mathcal{S}_{f}\left(h\right)$
in order to obtain a normalizing change of variables.\\

In case (b), for the same reasons, the order increases under the action
of $\mathcal{L}_{{f}}$ on the subspace $\big\{{h}\in\mathcal{L}_{k}:\ \mathrm{ord}({h})\geq\mathrm{ord}({g})\big\}\subseteq\mathfrak{L}_{>\mathrm{id}}$.
Indeed, for every ${h}\in\big\{{h}\in\mathcal{L}_{k}:\ \mathrm{ord}({h})\geq\mathrm{ord}({g})\big\}$,
analyzing \eqref{eq:operators S and T}, we obtain that 
\begin{equation}
\mathrm{ord}(\mathcal{L}_{{f}}({h}))\geq\mathrm{ord}({h})+\mathrm{ord}({g})-(1,\boldsymbol{0}_{k}).\label{eq:visi}
\end{equation}
Since $(1,\boldsymbol{0}_{k})<\mathrm{ord}({g})$, the order strictly
increases by the constant amount $\mathrm{ord}({g})-(1,\boldsymbol{0}_{k})>\mathbf{0}_{k+1}$.

Notice that $\mathrm{ord}_{z}({g})=1$ in case (b). Hence, for a general
${h}$, in \eqref{eq:ell} it may happen that 
\begin{align}
\mathrm{ord}_{z}(\mathcal{L}_{{f}}({h})) & =\mathrm{ord}_{z}({h}). \nonumber 
\end{align}
Thus the increase of order may not occur for $\mathrm{ord}_{z}$,
whereas it does occur for the lexicographic order $\mathrm{ord}$,
thanks to \eqref{eq:visi}.\\

For this reason, in case (b) we apply the following two-step algorithm
to reduce $f$ to $f_{0}$:
\begin{enumerate}[1., font=\textup, nolistsep, leftmargin=0.6cm]
\item \emph{Prenormalization}: this step consists in eliminating the leading
block of $f-f_{0}$ (in this case it is the block of order $1$ in $z$), thus leading to a prenormalized transseries.
\item \emph{Normalization}:\emph{ }we can now apply case (a) to the prenormalized
transseries obtained in the previous step.
\end{enumerate}

\subsection{Spaces of blocks}

\label{subsec:spaces blocks}\

\subsubsection{Spaces of blocks}

\label{subsec:spaces blocks_first}

Let $k\in\mathbb{N}_{\ge1}$. In view of case (b) of the proof of
the theorem, we introduce several families of spaces: $\mathcal{B}_{m}$,
$\mathcal{B}_{m}^{+}$ and $\mathcal{B}_{\ge m}^{+}$, for $1\leq m\leq k$.\\

a) \emph{The spaces $\mathcal{B}_{m}$ }(\emph{$1\leq m\leq k$}).
Let $\mathcal{B}_{m}\subseteq\mathcal{L}_{k}$ be the set of all transseries:
\begin{align}
R &=\sum_{\boldsymbol{n}\in\mathcal{S}}a_{\boldsymbol{n}}\boldsymbol{\ell}_{m}^{n_{m}}\cdots\boldsymbol{\ell}_{k}^{n_{k}}, \nonumber 
\end{align}
where $\boldsymbol{n}=(n_{m},\ldots,n_{k})\in\mathbb{Z}^{k-m+1}$
and $\mathcal{S}\subseteq\mathbb{Z}^{k-m+1}$ is well-ordered. The
order $\mathrm{ord}(R)$ is a tuple $(\mathbf{0}_{m},n_{m},\ldots,n_{k})\in\mathbb{R}_{\geq0}\times\mathbb{Z}^{k}$. 

We define \emph{the order in $\boldsymbol{\ell}_{m}$} \emph{of ${R}\in\mathcal{B}_{m}$
}as the minimal exponent of $\boldsymbol{\ell}_{m}$ in $R$,
and we denote it by 
\[
\mathrm{ord}_{\boldsymbol{\ell}_{m}}({R}).
\]
In the sequel, we will use the acronym $\mathrm{h.o.b.}(\boldsymbol{\ell}_{m})$,
for \textit{higher order blocks in} $\boldsymbol{\ell}_{m}$. The
sets $\mathcal{B}_{m}$ ($m\in\mathbb{N}_{\geq 1}$) are real vector spaces,
and $\mathcal{B}_{m+1}\subseteq\mathcal{B}_{m}$, for $1\leq m\leq k-1$.
\\

b) \emph{The spaces $\mathcal{B}_{m}^{+}$} (\emph{$1\leq m\leq k$})\emph{.}
Let 
\begin{align}
\mathcal{B}_{m}^{+} &:=\left\{ {R}\in\mathcal{B}_{m}:\mathrm{ord}_{\boldsymbol{\ell}_{m}}({R})>0)\right\} ,\ 1\leq m\leq k. \nonumber 
\end{align}

c) \emph{The spaces $\mathcal{B}_{\ge m}^{+}$ }(\emph{$1\leq m\leq k$})\emph{.}
Let $\mathcal{B}_{\geq m}^{+}\subseteq\mathcal{B}_{m}$ ($1\leq m\leq k$)
be the space of all blocks ${R}\in\mathcal{B}_{m}$ such that $\mathrm{ord}({R})>\mathbf{0}_{k+1}$.
Note that $\mathcal{B}_{\geq m}^{+}$ is also a real vector space,
which admits the following decomposition: 
\begin{align}
\mathcal{B}_{\geq m}^{+} &=\mathcal{B}_{m}^{+}\oplus\cdots\oplus\mathcal{B}_{k}^{+}. \nonumber 
\end{align}
Thus, every transseries ${R}\in\mathcal{B}_{\geq m}^{+}$ can be written
uniquely as 
\begin{align}
R &={R}_{m}+\cdots+{R}_{k}, \nonumber 
\end{align}
where ${R}_{i}\in\mathcal{B}_{i}^{+}$, $m\leq i\leq k$. \\

\subsubsection{\label{subsec:derivations Bm}Derivations on the spaces of blocks}

Note that $\frac{\mathrm{d}}{\mathrm{d}\boldsymbol{\ell}_{m}}$ (the
derivation with respect to the \emph{variable} $\boldsymbol{\ell}_{m}$,
applied term by term) turns $\mathcal{B}_{m}$ into a differential
algebra. In the following sections we will make an extensive use of
another derivation $D_{m}$ on $B_{m}$, defined as
\begin{equation}
D_{m}:=\boldsymbol{\ell}_{m}^{2}\frac{\mathrm{d}}{\mathrm{d}\boldsymbol{\ell}_{m}}.\label{eq:derivation Dm}
\end{equation}
This modified derivation will be better suited for our purposes, since
it increases the order in $\boldsymbol{\ell}_{m}$ by $1$ on $\mathcal{B}_{m}$.
Notice that the derivations $\frac{\mathrm{d}}{\mathrm{d}\boldsymbol{\ell}_{m}}$
and $D_{m}$ are superlinear, \emph{i.e.} they commute with infinite
sums. \\
Formula \eqref{eq:lien D1 Dm} explains the relation between derivations $D_{1}$ and $D_{m}$, for $2\leq m\leq k$.

\subsection{\label{subsec:gallery complete spaces}A gallery of complete spaces}

On the spaces of blocks $\mathcal{B}_{m}$, we define another distance
$d_{m}$ similar to the distance $d_{z}$ defined in \eqref{eq:distance dz}.
The function $d_{m}:\mathcal{B}_{m}\times\mathcal{B}_{m}\to\mathbb{R}_{\geq0}$
defined as 
\begin{eqnarray}
d_{m}\big({R}_{1},{R}_{2}\big)=\begin{cases}
2^{-\mathrm{ord}_{\boldsymbol{\ell}_{m}}({R}_{1}-{R}_{2})}, & {R}_{1}\neq{R}_{2}\\
0,\, & {R}_{1}={R}_{2}\quad\quad({R}_{1},\,{R}_{2}\in\mathcal{B}_{m})
\end{cases}
\end{eqnarray}
is a metric on the space $\mathcal{B}_{m}$. We call it the \emph{$m$-distance}. 
\begin{rem}
\label{rem:Dm contraction}Since the derivation $D_{m}$ defined in \eqref{eq:derivation Dm} increases the order in $\boldsymbol{\ell }_{m}$ by $1$ on the space $\mathcal{B}_{m}$, we
observe that, unlike $\frac{\mathrm{d}}{\mathrm{d}\boldsymbol{\ell}_{m}}$,
the derivation $D_{m}$ is a
$\frac{1}{2}$-contraction on $\left(\mathcal{B}_{m},d_{m}\right)$.
This property is used several times in the proof of the Main Theorem.
\end{rem}

In view of applying our version of the fixed point theorem (Proposition
\ref{prop:krasno fixed point}), we check in the next proposition
that all the metric spaces introduced so far are complete. 
\begin{prop}
\label{prop:complete spaces} The spaces $(\mathcal{L}_{k},d_{z})$,
$k\in\mathbb{N}$, the subspaces $\mathcal{L}_{k}^{\geq\beta}\subseteq\mathcal{L}_{k}$,
$\beta\geq0$, $k\in\mathbb{N}$, defined in \eqref{eq:L greater equal than beta},
and the spaces of blocks $(\mathcal{B}_{m},d_{m})$, $(\mathcal{B}_{m}^{+},d_{m})$,
$(\mathcal{B}_{\geq m}^{+},d_{m})$, $1\leq m\leq k$ (for $k\in \mathbb{N}_{\geq 1}$), defined in
Subsection~\ref{subsec:spaces blocks_first}, are complete. 
\begin{proof}
We prove that $\left(\mathcal{L}_{k},d_{z}\right)$ is a complete
space. The proof for the other spaces follows the same pattern. Suppose
that $\left({g}_{n}\right)$ is a Cauchy sequence in the space $\left(\mathcal{L}_{k},d_{z}\right)$.
Hence, for every $\alpha\in\mathbb{R}_{\geq0}$ there exists $n_{\alpha}\in\mathbb{N}$,
such that $p,q\ge n_{\alpha}$ implies $\mathrm{ord}_{z}\left(g_{p}-g_{q}\right)>\alpha$.
So, for every $n\geq n_{\alpha}$, every $0\leq\beta\leq\alpha$ and
every $\boldsymbol{m}\in\mathbb{Z}^{k}$: 
\begin{equation}
\left[g_{n}\right]_{\beta,\boldsymbol{m}}=\left[g_{n_{\alpha}}\right]_{\beta,\boldsymbol{m}}.\label{Eq20}
\end{equation}
We can then define an element $g\in\mathbb{R}^{\mathbb{R}_{\ge0}\times\mathbb{\mathbb{Z}}^{k}}$
by setting, for every $\left(\alpha,\boldsymbol{m}\right)\in\mathbb{R}_{\geq0}\times\mathbb{Z}^{k}$,
\begin{equation}
\left[g\right]_{\alpha,\boldsymbol{m}}=\left[g_{n_{\alpha}}\right]_{\alpha,\boldsymbol{m}}.\label{Eq22}
\end{equation}
It remains to prove that $g$ is indeed an element of $\mathcal{L}_{k}$,
and that $\left(g_{n}\right)\rightarrow g$ in $\left(\mathcal{L}_{k},d_{z}\right)$.
In order to prove that $g\in\mathcal{L}_{k}$, it is enough to prove
that $\mathrm{Supp}\left(g\right)$ is well-ordered. Let $A$ be a
nonempty subset of $\mathrm{Supp}\left({g}\right)$ and let $\left(\alpha,\boldsymbol{m}\right)\in A$.
It follows from the definition of ${g}$ that there exists $n_{\alpha}\in\mathbb{N}$
such that $\left[g\right]_{\alpha,\boldsymbol{m}}=\left[g_{n_{\alpha}}\right]_{\alpha,\boldsymbol{m}}$.
By \eqref{Eq20} we have 
\begin{equation}
\left[g\right]_{\beta,\boldsymbol{k}}=\left[g_{n_{\beta}}\right]_{\beta,\boldsymbol{k}}=\left[g_{n_{\alpha}}\right]_{\beta,\boldsymbol{k}},\label{Eq21}
\end{equation}
for every $0\leq\beta\leq\alpha$ and $\boldsymbol{k}\in\mathbb{Z}^{k}$.
From \eqref{Eq21} we deduce that 
\[
\left\lbrace \left(\beta,\mathbf n\right)\in A:\left(\beta,\mathbf n\right)\leq\left(\alpha,\mathbf m\right)\right\rbrace \subseteq\mathrm{Supp}\left({g}_{n_{\alpha}}\right).
\]
Since $\mathrm{Supp}\left(g_{n_{\alpha}}\right)$ is well-ordered,
the set $A$ admits a minimum element $\mathrm{min}A$. This implies
that $\mathrm{Supp}\left({g}\right)$ is a well-ordered subset of
$\mathbb{R}_{\geq0}\times\mathbb{Z}^{k}$. 

Finally, it follows easily from \eqref{Eq20} and \eqref{Eq22} that
$\left(g_{n}\right)\to g$ in $\left(\mathcal{L}_{k},d_{z}\right)$.
\end{proof}
\end{prop}

\subsection{A list of useful formulas}

Several results of the following sections make use of elementary formulas
involving the various notions of orders introduced above, as well
as the various distances and the derivatives $\frac{\mathrm{d}}{\mathrm{d}z}$,
$\frac{\mathrm{d}}{\mathrm{d}\boldsymbol{\ell}_{m}}$ and $D_{m}$
($m\in\mathbb{N}_{\geq 1}$). We give a list of some of them here, in order
to refer to them when needed. We omit their proofs, which consist
in straightforward computations.

The reader can skip this technical section at first reading. Here,
$\mu>0$, $H\in\mathcal{B}_{1}$, $G\in\mathcal{B}_{1}$ and $n_{1},\ldots,n_{k}\in\mathbb{Z}$, $k\in\mathbb N_{\geq 1}$. 
\begin{enumerate}[\(\bullet\), font=\textup, nolistsep, leftmargin=0.6cm]
\item 
\begin{align*}
\frac{\mathrm{d}}{\mathrm{d}z}z^{\mu} & \boldsymbol{\ell}_{1}^{n_{1}}\cdots\boldsymbol{\ell}_{k}^{n_{k}}\\
= & z^{\mu-1}\left(\mu\boldsymbol{\ell}_{1}^{n_{1}}\cdots\boldsymbol{\ell}_{k}^{n_{k}}+n_{1}\boldsymbol{\ell}_{1}^{n_{1}+1}\boldsymbol{\ell}_{2}^{n_{2}}\cdots\boldsymbol{\ell}_{k}^{n_{k}}+\cdots+n_{k}\boldsymbol{\ell}_{1}^{n_{1}+1}\boldsymbol{\ell}_{2}^{n_{2}+1}\cdots\boldsymbol{\ell}_{k}^{n_{k}+1}\right).
\end{align*}
\item ~
\begin{equation}
D_{1}\left(H\right)=\boldsymbol{\ell}_{1}\cdots\boldsymbol{\ell}_{m-1}D_{m}\left(H\right),\ H\in\mathcal B_m,\ m\geq 2.\label{eq:lien D1 Dm}
\end{equation}
\item If $\mathrm{ord}\left(H\right)=\left(0,\ldots0,n_{m},\ldots n_{k}\right)$, $1\leq m\leq k$,
with $n_{m}\ne0$, then 
\begin{equation}
\mathrm{ord}\left(D_{1}\left(H\right)\right)=\left(1,\ldots1,n_{m}+1,n_{m+1},\ldots,n_{k}\right).\label{eq:ord D1 H}
\end{equation}
\item \,
\begin{align*}
\left(zH\right)' & =H+D_{1}\left(H\right),\thinspace\left(zH\right)''=z^{-1}\left(D_{1}\left(H\right)+D_{1}^{2}\left(H\right)\right)=z^{-1}\left(D_{1}\left(H\right)+C_{2}\left(H\right)\right),\cdots,\\
\left(zH\right)^{\left(i\right)} & =z^{-\left(i-1\right)}\left(\left(i-2\right)!\left(-1\right)^{i}D_{1}\left(H\right)+C_{i}\left(H\right)\right),\thinspace i\ge2,
\end{align*}
where $C_{i}\colon\left(\mathcal{B}_{1},d_{1}\right)\rightarrow\left(\mathcal{B}_{1},d_{1}\right)$ for $i\geq 2$
is a linear $\frac{1}{4}$-contraction. More generally, if $\mu\in\mathbb{R}_{>1}$:
\begin{align}
\left(z^{\mu}H\right)' & =z^{\mu-1}\left(\mu H+D_{1}\left(H\right)\right)=z^{\mu-1}\left(\mu H+C_{1}\left(H\right)\right),\nonumber \\
\left(z^{\mu}H\right)'' & =z^{\mu-2}\left(\mu\left(\mu-1\right)H+C_{2}\left(H\right)\right),\cdots,\nonumber \\
\left(z^{\mu}H\right)^{\left(i\right)} & =z^{\mu-i}\left(\mu\left(\mu-1\right)\cdots\left(\mu-i+1\right)H+C_{i}\left(H\right)\right),\label{eq:i-th derivative block}
\end{align}
where $C_{i}\colon\left(\mathcal{B}_{1},d_{1}\right)\rightarrow\left(\mathcal{B}_{1},d_{1}\right)$
for $i\ge1$ is a linear $\frac{1}{2}$-contraction. \\

\item Let $H\in\mathcal{B}_{1}$ and $\lambda>0$. Then: 
\begin{equation}
	H(\lambda z)=H+\log\lambda\cdot D_{1}(H)+C(H),\label{eq:H applied to lambda z}
\end{equation}
where $C:\left(\mathcal{B}_{1},d_{1}\right)\to\left(\mathcal{B}_{1},d_{1}\right)$
is a linear $\frac{1}{4}$-contraction.\\

As an example, let us give a proof of formula \eqref{eq:H applied to lambda z}. Note that 
\[
\boldsymbol{\ell}_{1}(\lambda z)=\boldsymbol{\ell}_{1}+\log\lambda\cdot\boldsymbol{\ell}_{1}^{2}+\mathrm{h.o.b.}(\boldsymbol{\ell}_{1}).
\]
Since, for $k\in\mathbb{N}$, $\boldsymbol{\ell}_{k+1}(\lambda z)=\boldsymbol{\ell}_{1}\big(\boldsymbol{\ell}_{k}(\lambda z)\big)$,
we easily see by induction that 
\begin{align*}
\boldsymbol{\ell}_{m}\left(\lambda z\right) & =\boldsymbol{\ell}_{m}+\log\lambda\cdot\boldsymbol{\ell}_{1}\cdots\boldsymbol{\ell}_{m-1}\boldsymbol{\ell}_{m}^{2}+\mathrm{h.o.b.}(\boldsymbol{\ell}_{1}),\ m\in\mathbb{N}.
\end{align*}
Therefore, 
\begin{align}
\boldsymbol{\ell}_{m}^{n}\left(\lambda z\right) & =\boldsymbol{\ell}_{m}^{n}+n\log\lambda\cdot\boldsymbol{\ell}_{1}\cdots\boldsymbol{\ell}_{m-1}\boldsymbol{\ell}_{m}^{n+1}+\mathrm{h.o.b.}(\boldsymbol{\ell}_{1})\nonumber \\
& =\boldsymbol{\ell}_{m}^{n}+\log\lambda\cdot D_{1}\left(\boldsymbol{\ell}_{m}^{n}\right)+\mathrm{h.o.b.}(\boldsymbol{\ell}_{1}),\ m\in\mathbb{N},\ n\in\mathbb{Z}.\label{eq:lmsupn applied to lambda z-1}
\end{align}
Using \eqref{eq:lmsupn applied to lambda z-1}, we now verify \eqref{eq:H applied to lambda z}
for the terms $H=a_{n_{1},\ldots,n_{k}}\boldsymbol{\ell}_{1}^{n_{1}}\ldots\boldsymbol{\ell}_{k}^{n_{k}}\in\mathcal{B}_{1}$,
$n_{i}\in\mathbb{Z}$, $i=1,\ldots,k$, $a_{n_{1},\ldots,n_{k}}\in\mathbb{R}$.
Finally, for general blocks $H\in\mathcal{B}_{1}$, \eqref{eq:H applied to lambda z}
holds, after regrouping the terms, by superlinearity of the derivative
$D_{1}$. \\

\item Using \eqref{eq:i-th derivative block} and \eqref{eq:H applied to lambda z}, for $\lambda\in\mathbb{R}_{>0}$:
\begin{align}
&\sum_{i\ge1}\frac{\left(zH\right)^{\left(i\right)}\left(\lambda z\right)}{i!}\left(zG\right)^{i} \nonumber\\
&=zGH+\left(\log\lambda\right)zGD_{1}\left(H\right)+zD_{1}\left(H\right)\left(G+\lambda\sum_{i\ge2}\dfrac{\left(-1\right)^{i}}{i\left(i-1\right)}\left(\dfrac{G}{\lambda}\right)^{i}\right)+zC\left(H\right)\nonumber\\
&=z\left(GH+\left(\log\lambda\right)GD_{1}\left(H\right)+\lambda D_{1}\left(H\right)\left(1+\frac{G}{\lambda}\right)\log\left(1+\frac{G}{\lambda}\right)+C\left(H\right)\right), \label{eq:formule 4 dino}
\end{align}
where $C\colon\left(\mathcal{B}_{1},d_{1}\right)\rightarrow\left(\mathcal{B}_{1},d_{1}\right)$
is a linear $\frac{1}{4}$-contraction. \\

\item By \eqref{eq:i-th derivative block} and \eqref{eq:H applied to lambda z}, after regrouping:
\begin{align}
\sum_{i\ge1}\frac{\left(z^{\mu}H\right)^{\left(i\right)}\left(\lambda z\right)}{i!} \left(zG\right)^{i}&=z^{\mu}\lambda^{\mu}H\sum_{i\ge1}\binom{\mu}{i}\left(\frac{G}{\lambda}\right)^{i}+z^{\mu}C\left(H\right)\label{eq:taylor block lambda z}\nonumber\\
 & =z^{\mu}\left(\lambda^{\mu}H\left(\left(1+\frac{G}{\lambda}\right)^{\mu}-1\right)+C\left(H\right)\right),
\end{align}
where $C\colon\left(\mathcal{B}_{1},d_{1}\right)\rightarrow\left(\mathcal{B}_{1},d_{1}\right)$
is a linear $\frac{1}{2}$-contraction.
\end{enumerate}

\section{\label{sec:proof theorem}The proof of the Main Theorem}

Let $k\in\mathbb{N}$ and $f\in\mathcal{L}_{k}^{H}$ be hyperbolic.
We write, according to\,\eqref{eq:decomp f0 plus g}, 
\[
{f}={f}_{0}+{g}.
\]
Let ${g}=z^{\beta}{R}_{\beta}+\mathrm{h.o.b.}(z)$, where $\beta=\mathrm{ord}_{z}\left(g\right)$,
${R}_{\beta}\in\mathcal{B}_{1}$ and $\mathrm{ord}(z^{\beta}{R}_{\beta})>\mathbf{1}_{k+1}$.
We distinguish two cases:
\begin{enumerate}[1., font=\textup, nolistsep, leftmargin=0.6cm]
\item Case (a) : $\beta>1$. 
\item Case (b) : $\beta=1$. 
\end{enumerate}
The proof of the first part of the Main Theorem is done in Subsections~\ref{subsec:proof case a}
(case (a)) and\,\ref{subsec:proof case b} (case (b)). The second
part (minimality of the formal normal form) is proven in Subsection~\ref{subsec:minimality normal forms}.
Finally, the third part (convergence of the generalized Koenigs sequence
to the normalizing change of variables) is proven in Subsection~\ref{subsec:convergence koenigs}.

\subsection{Proof of case (a): $\beta>1$}

\label{subsec:proof case a}

The purpose of the following lemma is to verify that the operators
$\mathcal{T}_{f}$ and $\mathcal{S}_{f}$ satisfy the hypotheses of
Proposition \ref{prop:krasno fixed point}.
\begin{lem}
\label{lem:properties Sf Tf }Let $k\in\mathbb{N}$ and ${f}(z)=\lambda z+\mathrm{h.o.t.}\in\mathcal{L}_{k}^{H}$,
with $0<\lambda<1.$ Let ${f}_{0}$ and ${g}$ be as in decomposition
\eqref{eq:decomp f0 plus g} and $\beta:=\mathrm{ord}_{z}({g})>1$.
Let $\mathcal{T}_{{f}}$ and $\mathcal{S}_{{f}}$ be the operators
defined in \eqref{eq:operators S and T}. Then: 
\begin{enumerate}[1., font=\textup, nolistsep, leftmargin=0.6cm]
\item $\mathcal{L}_{k}^{\ge\beta}$ is invariant under $\mathcal{T}_{{f}}$
and $\mathcal{S}_{{f}}\,$,
\item $\mathcal{S}_{{f}}$ is a $\frac{1}{2^{\beta-1}}$-contraction on
the space $\big(\mathcal{L}_{k}^{\ge\beta},d_{z}\big)$,
\item $\mathcal{T}_{{f}}$ is an isometry and a surjection on the space
$\big(\mathcal{L}_{k}^{\geq\beta},d_{z}\big)$.
\end{enumerate}
The same holds for the spaces $\mathcal{L}_{m}^{\geq\beta}$, in place
of $\mathcal{L}_{k}^{\ge\beta}$, for all $m\geq k$. 
\end{lem}

\begin{proof}
\,

\emph{Proof of 1.} Note that ${g}\in\mathcal{L}_{m}^{\geq\beta}$,
for every $m\geq k$. The invariance of $\mathcal{L}_{m}$, $m\geq k$,
and of the subspaces $\mathcal{L}_{m}^{\geq\beta}$ under $\mathcal{T}_{{f}}$
and $\mathcal{S}_{{f}}$ follows easily from \eqref{eq:operators S and T}
and Proposition \ref{prop:formal Taylor}. \\

\emph{Proof of 2.} Thanks to Proposition \ref{prop:formal Taylor}, the operator
$\mathcal{S}_{f}$ defined in \eqref{eq:operators S and T} can be
expanded as 
\[
\mathcal{S}_{f}\left(h\right)=\dfrac{1}{\lambda}g+\frac{1}{\lambda}\sum_{i\ge1}\frac{h^{\left(i\right)}\circ f_{0}}{i!}g^{i}-\frac{1}{\lambda}\sum_{i\ge2}\frac{g_{0}^{\left(i\right)}}{i!}h^{i},\thinspace h\in\mathcal{L}_{>\mathrm{id}}.
\]Let ${h}_{1},{h}_{2}\in\mathcal{L}_{m}^{\geq\beta}$,
$m\geq k$. Then $\mathrm{ord}_{z}({h}_{1}),\mathrm{ord}_{z}({h}_{1})\geq\beta$.
Since $\beta=\mathrm{ord}_{z}({g})$, we obtain 
\begin{align}
\begin{split}\mathrm{ord}_{z}\Bigg(\sum_{i\geq1}\frac{{h}_{1}^{(i)}\circ{f}_{0}}{i!}g^{i}-\sum_{i\geq1}\frac{{h}_{2}^{(i)}\circ{f}_{0}}{i!}g^{i}\Bigg) & =\mathrm{ord}_{z}\Bigg(\sum_{i\geq1}\frac{\left({h}_{1}-{h}_{2}\right)^{(i)}\circ{f}_{0}}{i!}g^{i}\Bigg)\\
 & =\mathrm{ord}_{z}({h}_{1}-{h}_{2})+\beta-1,
\end{split}
\label{eq:one}
\end{align}
and 
\begin{eqnarray}
\begin{split}\mathrm{ord}_{z}\Bigg(\sum_{i\geq2}\frac{{g}_{0}^{(i)}}{i!}{h}_{1}^{i}-\sum_{i\geq2}\frac{{g}_{0}^{(i)}}{i!}{h}_{2}^{i}\Bigg) & =\mathrm{ord}_{z}\Bigg(\sum_{i\geq2}\frac{{g}_{0}^{(i)}}{i!}\left({h}_{1}-{h}_{2}\right)\Big(\sum_{j=0}^{i-1}{h}_{1}^{j}{h}_{2}^{i-j-1}\Big)\Bigg)\\
 & \geq\mathrm{ord}_{z}(h_{1}-h_{2})+\beta-1.
\end{split}
\label{eq:two}
\end{eqnarray}
The equations \eqref{eq:one} and \eqref{eq:two} imply that $\mathcal{S}_{f}$
is a $\frac{1}{2^{\beta-1}}$-contraction on the space $\left(\mathcal{L}_{k},d_{z}\right)$,
as well as on the spaces $\mathcal{L}_{m}^{\geq\beta}$, $m\geq k$.
\\

\emph{Proof of 3.} We first prove that $\mathcal{T}_{{f}}$ is an
isometry on $\left(\mathcal{L}_{m}^{\geq\beta},d_{z}\right)$, $m\geq k$.
As we did for the operator $\mathcal{S}_{f}$, we use Proposition~\ref{prop:formal Taylor}
to expand $\mathcal{T}_{f}$ from \eqref{eq:operators S and T} in the following way:
\begin{align}
	\mathcal{T}_{f}\left(h\right) &=h-\frac{1}{\lambda}h\left(\lambda z\right)-\frac{1}{\lambda}\sum_{i\ge1}\frac{h^{\left(i\right)}\left(\lambda z\right)}{i!}g_{0}^{i}+\frac{1}{\lambda}g_{0}'\cdot h,\thinspace h\in\mathcal{L}_{>\mathrm{id}}. \label{EqTf}
\end{align}
Let ${h}=z^{\alpha}H_{\alpha}+\mathrm{h.o.b.}(z)\in\mathcal{L}_{m}^{\geq\beta},\ m\geq k,$
where ${H}_{\alpha}\in\mathcal{B}_{1}$, $\alpha\geq\beta$. Analyzing
the orders of the terms of $\mathcal{T}_{{f}}({h})$ in \eqref{EqTf},
using \eqref{eq:H applied to lambda z} to expand $h\left(\lambda z\right)$, the fact that $\mathrm{ord}({g}_{0})>(1,0,\ldots,0)_{m+1}$ and $\lambda\neq 1$, $\alpha >1$, we
conclude that
\[
\mathrm{ord}(\mathcal{T}_{{f}}({h}))=\mathrm{ord}(z^{\alpha}{H}_{\alpha}).
\]
Hence, 
\begin{equation}
\mathrm{ord}_{z}(\mathcal{T}_{{f}}({h}))=\mathrm{ord}_{z}(z^{\alpha}{H}_{\alpha})=\mathrm{ord}_{z}(h).\label{eq:eqo}
\end{equation}
Therefore, $\mathcal{T}_{{f}}$ is a linear isometry on the spaces
$\big(\mathcal{L}_{m}^{\geq\beta},d_{z}\big)$, $m\geq k$.\\

It remains to prove that $\mathcal{T}_{{f}}:\mathcal{L}_{k}^{\geq\beta}\to\mathcal{L}_{k}^{\geq \beta}$
is a surjection, and that this also holds if we replace $\mathcal{L}_{k}^{\geq \beta}$
by $\mathcal{L}_{m}^{\geq\beta}$, $m\geq k$. Due to the superlinearity
of $\mathcal{T}_{{f}}$, it is sufficient to prove that, for every
block $z^{\gamma}{M}_{\gamma}\in\mathcal{L}_{k}^{\ge\beta},\ {M}_{\gamma}\in\mathcal{B}_{1}$,
there exists a block $z^{\alpha}{H}_{\alpha}\in\mathcal{L}_{k}^{\ge\beta}$, $H_{\alpha }\in \mathcal{B}_{1}$, such that 
\begin{equation}
\mathcal{T}_{f}\big(z^{\alpha}H_{\alpha}\big)=z^{\gamma}M_{\gamma}.\label{eq:Tf surjection}
\end{equation}
The idea there is to prove the existence of a solution to \eqref{eq:Tf surjection}
by reformulating this equation as a fixed point equation for a suitable
contraction on the complete space $\left(\mathcal{B}_{1},d_{1}\right)$. 

First, as $\mathcal{T}_{f}$ is an isometry, $\alpha=\gamma$. Write
${g}_{0}=z{Q}$, with ${Q}\in\mathcal{B}_{\geq 1}^{+}$. Using \eqref{eq:i-th derivative block},
\eqref{eq:taylor block lambda z} and \eqref{eq:H applied to lambda z},
we regroup the elements of the left-hand side $\mathcal{T}_{{f}}\big(z^{\gamma}{H}_{\gamma}\big)$
of \eqref{eq:Tf surjection} as 
\begin{align*}
\lambda z^{\gamma}H_{\gamma}-\big(z^{\gamma}H_{\gamma}\big)(\lambda z) & =\big(\lambda-\lambda^{\gamma}\big)z^{\gamma}H_{\gamma}-\lambda^{\gamma}z^{\gamma}\big(\log\lambda\cdot D_{1}(H_{\gamma})+\mathcal{C}(H_{\gamma})\big),\\
\big(zQ\big)'\cdot z^{\gamma}{H}_{\gamma} & =z^{\gamma}\big(Q+D_{1}(Q)\big)H_{\gamma},\\
\sum_{i\geq1}\frac{(z^{\gamma}H_{\gamma})^{(i)}(\lambda z)}{i!}(zQ)^{i} & =z^{\gamma}\lambda^{\gamma}H_{\gamma}\sum_{i\geq1}{\gamma \choose i}\left(\frac{Q}{\lambda}\right)^{i}+z^{\gamma}\mathcal{K}({H}_{\gamma}),
\end{align*}
where $\mathcal{C}$ from \eqref{eq:H applied to lambda z} is a linear $\frac{1}{4}$-contraction
on $(\mathcal{B}_{1},d_{1})$ and $\mathcal{K}$ from \eqref{eq:taylor block lambda z}
is a linear $\frac{1}{2}$-contraction on $\left(\mathcal{B}_{1},d_{1}\right)$:
\[
\mathrm{ord}_{\boldsymbol{\ell}_{1}}(\mathcal{K}(H_{\gamma}))\geq\mathrm{ord}_{\boldsymbol{\ell}_{1}}(H_{\gamma})+1.
\]
Finally, the operators $D_{1},\ \mathcal{C}$ and $\mathcal{K}$ do
not decrease the powers of the variables $\boldsymbol{\ell}_{m}$,
for $1\leq m\leq k$. Hence, after dividing by $z^{\gamma}$, these
identities allow to rewrite \eqref{eq:Tf surjection} as the following
fixed point equation: 
\[
H_{\gamma}=\mathcal{S}_{1}(H_{\gamma}),
\]
where $\mathcal{S}_{1}:\mathcal{B}_{1}\to\mathcal{B}_{1}$ is the
operator defined by 
\begin{equation}
\mathcal{S}_{1}(H):=\frac{\lambda^{\gamma}\big(\log\lambda\cdot D_{1}(H)+\mathcal{C}(H)\big)-H\,D_{1}(Q)+\mathcal{K}(H)+M_{\gamma}}{\lambda-\lambda^{\gamma}+Q-\lambda^\gamma\sum_{i\geq1}{\gamma \choose i}\left(\frac{Q}{\lambda}\right)^{i}},\ {H}\in\mathcal{B}_{1}.\label{eq:S1}
\end{equation}
Note that $\gamma\geq\beta>1$, so $\lambda-\lambda^{\gamma}\neq0.$
Hence, thanks to \eqref{eq:ord D1 H}, $\mathcal{S}_{1}$ is a linear $\frac{1}{2}$-contraction
on the space $\big(\mathcal{B}_{1},d_{1}\big)$, which is complete
by Proposition~\ref{prop:complete spaces}. It follows from the Banach
fixed point theorem that $\mathcal{S}_{1}$ has a unique fixed point
in $\mathcal{B}_{1}$, so that the block $z^{\gamma}M_{\gamma}$ has
a unique preimage $z^{\gamma}H_{\gamma}\in\mathcal{L}_{m}^{\ge\beta}$
by $\mathcal{T}_{f}$. 

Finally, thanks to the superlinearity of $\mathcal{T}_{{f}}$, we
conclude that $\mathcal{T}_{f}$ is surjective on $\mathcal{L}_{m}^{\ge\beta}$.
\end{proof}
We can now conclude the proof of the Main Theorem in case (a) .
\begin{proof}[Proof of case (a)]
By Proposition~\ref{prop:complete spaces}, the space $\mathcal{L}_{k}^{\ge\beta}$
is complete. By Lemma~\ref{lem:properties Sf Tf } and Proposition~\ref{prop:krasno fixed point},
the equation $\mathcal{T}_{{f}}({h})=\mathcal{S}_{{f}}({h})$ has
a unique solution ${h}\in\mathcal{L}_{k}^{\ge\beta}$ . Since $\beta>1$,
it follows that ${h}\in\mathfrak{L}_{>\mathrm{id}}$.\\

We now prove the uniqueness of the solution of 
\begin{equation}
\mathcal{T}_{f}(h)=\mathcal{S}_{f}(h)\label{eq:temp}
\end{equation}
in the larger space\emph{ $\mathfrak{L}_{>\mathrm{id}}$}. Suppose
that there exists another solution ${h}_{1}\in\mathcal{L}_{m}\cap\mathfrak{L}_{>\mathrm{id}}$,
for some $m\geq k$, of \eqref{eq:temp}, where ${h}_{1}\neq{h}$.

We prove that $\mathrm{ord}_{z}({h}_{1})\geq\beta$. To this end we
introduce the operators 
\[
\widetilde{\mathcal{S}}_{f}\left(h\right):=\frac{1}{\lambda}\left(g+\left(h\circ f-h\circ f_{0}\right)\right)
\]
and 
\[
\widetilde{\mathcal{T}}_{f}\left(h\right):=\dfrac{1}{\lambda}\left(\left(\lambda h-h\left(\lambda z\right)\right)-\left(h\circ f_{0}-h\left(\lambda z\right)\right)+g_{0}'\cdot h+\left(g_{0}\left(\mathrm{id}+h\right)-g_{0}-g_{0}'\cdot h\right)\right)
\]
obtained by moving the last term of $\mathcal{S}_{f}\left(h\right)$
to $\mathcal{T}_{f}\left(h\right)$ (in \eqref{eq:operators S and T}). So we have \begin{equation}\label{eq:lst}\widetilde{\mathcal{S}}_{f}\left(h_{1}\right)=\widetilde{\mathcal{T}}_{f}\left(h_{1}\right).\end{equation}
Since $\beta=\mathrm{ord}_{z}({g})$ and $\mathrm{ord}_{z}({h}_{1})\geq1$
for ${h}_{1}\in\mathfrak{L}_{>\mathrm{id}}$, by Taylor expansion (Proposition \ref{prop:formal Taylor}) it follows that 
\begin{equation*}
\mathrm{ord}_{z}\big(\widetilde{\mathcal{S}}_{f}(h_{1})\big)\geq\min\big\{\beta,\mathrm{ord}_{z}{h}_{1}+\beta-1\big\}=\beta.
\end{equation*}
On the other hand, it can be seen as in \eqref{eq:eqo} in the proof
of Lemma~\ref{lem:properties Sf Tf } that we have the identity
\begin{equation*}
\mathrm{ord}_{z}(\widetilde{\mathcal{T}}_{f}(h_{1}))=\mathrm{ord}_{z}(h_{1}).
\end{equation*}
Comparing the orders of the left and the right hand sides of \eqref{eq:lst},
we obtain 
\[
\mathrm{ord}_{z}({h}_{1})\geq\beta.
\]
That is, ${h}_{1}\in\mathcal{L}_{m}^{\geq\beta}$. Recall that the
space $\mathcal{L}_{m}^{\geq\beta}$, $m\in\mathbb{N}$, is complete
by Proposition~\ref{prop:complete spaces}. Hence Lemma~\ref{lem:properties Sf Tf }
and Proposition~\ref{prop:krasno fixed point} give the uniqueness
of the solution of \eqref{eq:temp} in $\mathcal{L}_{m}^{\geq\beta}$.
Now, since $\mathcal{L}_{k}^{\geq\beta}\subseteq\mathcal{L}_{m}^{\geq\beta}$,
both ${h}$ and ${h}_{1}$ belong to $\mathcal{L}_{m}^{\geq\beta}$,
which contradicts the uniqueness of the solution in $\mathcal{L}_{m}^{\geq\beta}$.

Finally, by Proposition~\ref{prop:fixed point equivalent conjugacy},
the uniqueness of the solution ${h}$ of the equation $\mathcal{S}_{{f}}({h})=\mathcal{T}_{{f}}({h})$
in $\mathfrak{L}_{>\mathrm{id}}$ implies the uniqueness of the normalizing
change of variables $\varphi=\mathrm{id}+h$ in the space $\mathfrak{L}^{0}$.
This proves case (a).
\end{proof}

\subsection{Proof of case (b):\ $\beta=1$}

\label{subsec:proof case b}\

We proceed in two steps. First, we eliminate all the terms of the
leading block
of ${f}$ that can be eliminated (namely, those which do not belong
to the normal form $f_{0}$ from \eqref{eq:normal form}). We call
this procedure the \emph{prenormalization} of ${f}$. After prenormalization,
we can apply the results of case (a).

\subsubsection{Prenormalization}

We prove the existence and the uniqueness, up to blocks of order in
$z$ strictly higher than $1$, of a change of variables ${\varphi}_{0}\in\mathfrak{L}^{0}$,
such that: 
\begin{align}
 & {\varphi}_{0}\circ{f}\circ{\varphi}_{0}^{-1}={f}_{0}+\mathrm{h.o.b.}(z),\label{eq:equation conjugacy phi0}\\
 & {\varphi}_{0}\circ{f}={f}_{0}\circ{\varphi}_{0}+\mathrm{h.o.b.}(z),\nonumber 
\end{align}
where ${f}_{0}$ is the formal normal form \eqref{eq:normal form}.
Moreover, we prove that ${\varphi}_{0}\in\mathcal{L}_{k}$. 

First, as in case (a), we rewrite the prenormalization equation \eqref{eq:equation conjugacy phi0}
as a fixed point equation (\eqref{eq:fixed point T0 and S0} in Lemma
\ref{lem:fixed point T0 S0}) on the spaces of blocks. Then, in Lemma~\ref{lem:hypotheses T0 S0},
we prove that the operators $\mathcal{T}_{0}$ and $\mathcal{S}_{0}$
(\eqref{eq:S0 and T0} in Lemma \ref{lem:fixed point T0 S0}) satisfy
all the assumptions of the fixed point Proposition~\ref{prop:krasno fixed point}.
This gives the existence of a unique fixed point, which solves the
prenormalization equation.
\begin{lem}
\label{lem:fixed point T0 S0} Let $k\in\mathbb{N}_{\geq 1}$ and $f(z)=\lambda z+\mathrm{h.o.t.}\in\mathcal{L}_{k}^{H}$,
with $0<\lambda<1.$ Write $f=f_{0}+g$ as in decomposition \eqref{eq:decomp f0 plus g},
with $\mathrm{ord}_{z}({g})=1$.

A transseries $\varphi_{0}=z+zH+\mathrm{h.o.b.}(z),\ \text{with\ }H\in\mathcal{B}_{\geq1}^{+},$
satisfies the prenormalization equation \eqref{eq:equation conjugacy phi0}
if and only if the block $H$ satisfies the equation 
\begin{equation}
\mathcal{T}_{0}(H)=\mathcal{S}_{0}\big(H\big),\label{eq:fixed point T0 and S0}
\end{equation}
where the operators $\mathcal{T}_{0},\,\mathcal{S}_{0}:\mathcal{B}_{\geq1}^{+}\to\mathcal{B}_{1}^{+}\subset\mathcal{B}_{\geq1}^{+}$,
are defined by: 
\begin{align}
\mathcal{T}_{0}(H) & :=\Big(-\lambda\log\lambda-(1+\log\lambda)R_{0}-\lambda\sum_{i\geq2}\frac{(-1)^{i}}{i(i-1)}\left(\frac{R_{0}}{\lambda}\right)^{i}\Big)D_{1}\big(H\big)+\label{eq:S0 and T0}\\
 & \qquad\qquad\qquad\qquad\qquad\qquad+D_{1}(R_{0})\Big(H+\sum_{i\geq2}\frac{(-1)^{i}}{i(i-1)}H^{i}\Big)-HR \, , \nonumber \\
\mathcal{S}_{0}(H) & :=\lambda\mathcal{C}(H)+\mathcal{K}_{1}(H)-\mathcal{K}_{2}(H)+\mathcal{K}_{3}(H)+{R}.\nonumber 
\end{align}
Here, ${R}\in\mathcal{B}_{1}^{+}\subseteq\mathcal{L}_{k}$ is defined
by ${g}=z{R}+\mathrm{h.o.b.}(z)$, and $R_{0}\in\mathcal{B}_{\geq1}^{+}\subseteq\mathcal{L}_{k}$
is defined by ${f}_{0}=\lambda z+zR_{0}$. The operators $\mathcal{C}$,\ $\mathcal{K}_{1},\,\mathcal{K}_{2},\,\mathcal{K}_{3}:\mathcal{B}_{\geq1}^{+}\to\mathcal{B}_{1}^{+}\subset\mathcal{B}_{\geq1}^{+}$
are suitable $\frac{1}{4}$-contractions in the metric $d_{1}$.
\end{lem}

The precise definition of these contraction operators will emerge
during the proof. 
\begin{proof}[Proof of Lemma~\ref{lem:fixed point T0 S0}]
Setting $\varphi_{0}=z+zH+\mathrm{h.o.b.}(z)$, where $H\in\mathcal{B}_{\ge1}^{+}$,
${f}_{0}=\lambda z+zR_{0}$ and ${f}={f}_{0}+z{R}+\mathrm{h.o.b.}(z)$,
we rewrite the prenormalization equation \eqref{eq:equation conjugacy phi0}
as an equation satisfied by $H$, $R$ and $R_{0}$. To this end,
we use Proposition \ref{prop:formal Taylor} to expand the compositions
in \eqref{eq:equation conjugacy phi0} and compare the leading blocks
(namely the blocks with $\mathrm{ord}_{z}$ equal to $1$) of both
sides of the equation. We obtain: 
\begin{align}
\lambda z{H}-\lambda z{H}(\lambda z)-\sum_{i\geq1} & \frac{\big(z{H}\big)^{(i)}(\lambda z)}{i!}\big(z{R_{0}}\big)^{i}\nonumber \\
 & =\sum_{i\geq1}\frac{\big(z{H}\big)^{(i)}\big(\lambda z+z{R_{0}}\big)}{i!}\big(z{R}\big)^{i}-\sum_{i\geq1}\frac{\big(z{R_{0}}\big)^{(i)}}{i!}\big(z{H}\big)^{i}+z{R}.\label{eq:S0 equals T0}
\end{align}
By \eqref{eq:H applied to lambda z} and \eqref{eq:formule 4 dino},
we have:
\begin{align}
 & \lambda zH-\lambda zH(\lambda z)=-\lambda\log\lambda\cdot zD_{1}(H)-\lambda z\mathcal{C}(H),\label{eq:S0 analysis}\\
 & \sum_{i\geq1}\frac{\big(z{H}\big)^{(i)}(\lambda z)}{i!}\big(z{R_{0}}\big)^{i}\nonumber\\
&\quad \qquad =zHR_{0}+zD_{1}(H)\Big((1+\log\lambda)R_{0}+\lambda\sum_{i\geq2}\frac{(-1)^{i}}{i(i-1)}\left(\frac{R_{0}}{\lambda}\right)^{i}\Big)+z\mathcal{K}_{1}(H),\nonumber 
\end{align}
where $\mathcal{C}$ and $\mathcal{K}_{1}$ are linear $\frac{1}{4}$-contractions
on the space $\big(\mathcal{B}_{\geq1}^{+},d_{1}\big)$. Moreover, since $\mathrm{ord}(R)>(0,\boldsymbol{1}_{k})$, by \eqref{eq:formule 4 dino}, it follows that:
\begin{align}
 & \sum_{i\geq1}\frac{\big(zR_{0}\big)^{(i)}}{i!}\big(zH\big)^{i}=zR_{0}{H}+zD_{1}(R_{0})\Big(H+\sum_{i\geq2}\frac{(-1)^{i}}{i(i-1)}H^{i}\Big)+z\mathcal{K}_{2}(H),\label{eq:T0 analysis}\\
 & \sum_{i\geq1}\frac{\big(zH\big)^{(i)}\big(\lambda z+zR_{0}\big)}{i!}\big(z{R}\big)^{i}=\big(zH\big)'\big(\lambda z+zR_{0}\big)z{R}+\sum_{i\geq2}\frac{\big(zH\big)^{(i)}\big(\lambda z+zR_{0}\big)}{i!}\big(z{R}\big)^{i}=\nonumber \\
 & \qquad\qquad\qquad\qquad\qquad\qquad=zH{R}+z\mathcal{K}_{3}(H),\nonumber 
\end{align}
where the operators $\mathcal{K}_{2}$ and $\mathcal{K}_{3}$ are
$\frac{1}{4}$-contractions on the space $\big(\mathcal{B}_{\geq1}^{+},d_{1}\big)$.

Now, by eliminating $z$, using \eqref{eq:S0 analysis} and \eqref{eq:T0 analysis} in
\eqref{eq:S0 equals T0}, we obtain: 
\begin{align}
\Big(-\lambda\log\lambda-(1+\log\lambda)R_{0}- & \lambda\sum_{i\geq2}\frac{(-1)^{i}}{i(i-1)}\left(\frac{R_{0}}{\lambda}\right)^{i}\Big)D_{1}(H)+\label{eq:S0 equals T0 developed}\\
 & +D_{1}(R_{0})\Big(H+\sum_{i\geq2}\frac{(-1)^{i}}{i(i-1)}H^{i}\Big)-H{R}=\nonumber \\
 & \qquad\qquad=\lambda\mathcal{C}(H)+\mathcal{K}_{1}(H)-\mathcal{K}_{2}(H)+\mathcal{K}_{3}(H)+{R}.\nonumber 
\end{align}
It follows that \eqref{eq:S0 equals T0 developed} is equivalent to
$\mathcal{S}_{0}(H)=\mathcal{T}_{0}(H)$, where $\mathcal S_0$ and $\mathcal T_0$ are defined in \eqref{eq:S0 and T0}. 
\end{proof}
\begin{lem}
\label{lem:hypotheses T0 S0}Let $k\in\mathbb{N}_{\geq 1}$ and ${f}(z)=\lambda z+\mathrm{h.o.t.}\in\mathcal{L}_{k}^{H}$,
with $0<\lambda<1.$ Let ${f}_{0}$ and ${g}$ be as in decomposition
\eqref{eq:decomp f0 plus g}. Let $\mathrm{ord}_{z}({g})=1$. Let
the operators $\mathcal{T}_{0},\,\mathcal{S}_{0}:\mathcal{B}_{\geq1}^{+}\to\mathcal{B}_{1}^{+}$,
$\mathcal{B}_{\geq1}^{+}\subseteq\mathcal{L}_{k}$, be defined as
in \eqref{eq:S0 and T0}. Then: 
\begin{enumerate}[1., font=\textup, nolistsep, leftmargin=0.6cm]
\item $\mathcal{T}_{0}$ is a linear $\frac{1}{2}$-homothety on $\mathcal{B}_{\geq1}^{+}$
with respect to the metric $d_{1}$, 
\item $\mathcal{S}_{0}$ is a $\frac{1}{4}$-contraction on $\mathcal{B}_{\geq1}^{+}$
with respect to the metric $d_{1}$, 
\item $\mathcal{S}_{0}(\mathcal{B}_{\geq1}^{+})\subseteq\mathcal{T}_{0}(\mathcal{B}_{\geq1}^{+})$. 
\end{enumerate}
\end{lem}

\noindent The proof of the two first parts of Lemma \ref{lem:hypotheses T0 S0}
follows easily by observing the orders of the monomials involved in
the expansion of $\mathcal{S}_{0}\left(H\right)$ and $\mathcal{T}_{0}\left(H\right)$,
for $H\in\mathcal{B}_{\ge1}^{+}$. Note that $\mathrm{ord}(\mathcal T_0(H))=\mathrm{ord}(D_1(H)),\ H\in\mathcal B_1$. The third part is however more
subtle, and requires the resolution of some \emph{non linear differential
equations} in $\mathcal{B}_{1}$. Interestingly, the method we use
in this resolution is another application of our version of the fixed
point theorem (Proposition \ref{prop:krasno fixed point}). Hence
we devote Subsection \ref{subsec:differential equations on blocks}
to the description of this method. In Subsection \ref{subsec:end case b}
we prove Lemma \ref{lem:hypotheses T0 S0} and finish the proof of
case (b) of the Main Theorem.

\subsubsection{\label{subsec:differential equations on blocks}Differential equations
in spaces of blocks}

In this subsection we show how to use Proposition \ref{prop:krasno fixed point}
to solve various differential equations in spaces of blocks. We begin
by the following technical result. Recall the definition of \emph{$\mathcal{L}_{k}^{\infty}$
}in Subsection \eqref{subsec:Notation}.
\begin{lem}[\emph{Stability of spaces $\mathcal{L}_{k}^{\infty}$ under integration}]
\label{lem:primitive}\,$\ $
\begin{enumerate}[i), font=\textup, nolistsep,leftmargin=0.6cm]
\item $\int z^{-1}\boldsymbol{\ell}_{1}\cdots\boldsymbol{\ell}_{k}\ dz=-\boldsymbol{\ell}_{k+1}^{-1},\ k\in\mathbb{N}$, 
\item For $(\delta,m_{1},\ldots,m_{k})\in\mathbb{R}\times\mathbb{Z}^{k}$,
$(\delta,m_{1},\ldots,m_{k})\neq(-1,1,\ldots,1)$, 
\[
\int z^{\delta}\boldsymbol{\ell}_{1}^{m_{1}}\cdots\boldsymbol{\ell}_{k}^{m_{k}}\,dz\in\mathcal{L}_{k}^{\infty}.
\]
\end{enumerate}
\end{lem}

\begin{proof}
We prove the first part by the direct computation of the derivative
of $-\boldsymbol{\ell}_{k+1}^{-1}$.

For the second part, when $\delta\ne-1$, first we notice that the
series 
\[
G:=\sum_{i=0}^{\infty}\big(\frac{-1}{\delta+1}\big)^{i}D_{1}^{i}(\boldsymbol{\ell}_{1}^{m_{1}}\cdots\boldsymbol{\ell}_{1}^{m_{k}})
\]
 converges in the complete space $\left(\mathcal{B}_{1},d_{1}\right)$,
where $D_{1}$ is the derivation defined in \eqref{eq:derivation Dm}.
Indeed, $D_{1}$ is a $\frac{1}{2}$-contraction on $\mathcal{B}_{1}$
(see Remark \ref{rem:Dm contraction}). Secondly, using the first line of \eqref{eq:i-th derivative block} we prove that $\frac{z^{\delta+1}}{\delta+1}G$
is an antiderivative of $z^{\delta}\boldsymbol{\ell}_{1}^{m_{1}}\cdots\boldsymbol{\ell}_{k}^{m_{k}}$,
by computing its derivative term by term.

Similarly, if $\delta=-1$, suppose that $m_{1}=m_{2}=\cdots=m_{r-1}=1$
and $m_{r}\ne1$, for some $r\in\left\{ 1,\ldots,k\right\} $, and
let $K:=\boldsymbol{\ell}_{r+1}^{m_{r+1}}\cdots\boldsymbol{\ell}_{k}^{m_{k}}$
if $r<k$ (and $K=1$ if $r=k$). Since $D_{r+1}$ is a $\frac{1}{2}$-contraction
on the complete space $\left(\mathcal{B}_{r+1},d_{r+1}\right)$ (by
Remark \ref{rem:Dm contraction}), the series $\sum_{i=0}^{\infty}\frac{\left(-1\right)^{i}D_{r+1}^{i}\left(K\right)}{(m_{r}-1)^{i}}$
converges in this space. A direct computation, using \eqref{eq:lien D1 Dm}, shows that the derivative
of 
\begin{equation}\label{eq:cham}
\frac{\boldsymbol{\ell}_{r}^{m_{r}-1}}{m_{r}-1}\sum_{i=0}^{\infty}\frac{\left(-1\right)^{i}D_{r+1}^{i}\left(K\right)}{\left(m_{r}-1\right)^{i}}
\end{equation}
is $z^{-1}\boldsymbol{\ell}_{1}\cdots\boldsymbol{\ell}_{r-1}\boldsymbol\ell_r^{m_{r}}K$. Note that \eqref{eq:cham} remains in $\mathcal L_k^\infty$.
\end{proof}
In view of solving linear equations in spaces of transseries, we need
an exponential operator acting on them. Its existence is guaranteed
by the following classic lemma, the proof of which we omit (see for
example \cite{vdd_macintyre_marker:log-exp_series}).
\begin{lem}[Formal exponential in $\mathcal{L}_{k}$]
\label{lem:exponential} The \emph{formal exponential operator} $\exp:\{{f}\in\mathcal{L}_{k}:\mathrm{ord}({f})\geq\mathbf{0}_{k+1}\}\to\{{f}\in\mathcal{L}_{k}:\mathrm{ord}({f})\geq\mathbf{0}_{k+1}\}$,
defined by: 
\begin{equation}
\exp\left(c+{g}\right):=\exp c\cdot\Big(1+\sum_{i\geq1}\frac{{g}^{i}}{i!}\Big),\ \ \,c\in\mathbb{R},\ {g}\in\mathcal{L}_{k},\,\mathrm{ord}({g})>\mathbf{0}_{k+1},\label{eq:exp}
\end{equation}
converges in the weak topology in $\mathcal{L}_{k}$. Moreover, if
$ord_{z}({g})>0$, then the series \eqref{eq:exp} converges in the
(finer) valuation topology. 
\end{lem}

We consider now a first type of differential equation in the spaces
of blocks involved in the proof of case (b) of the Main Theorem. Lemma \ref{lem:differential equation Dm} is used for proving Lemma \ref{lem:differential equation D1}, which is the main result of Subsection \ref{subsec:differential equations on blocks}.
\begin{lem}
\label{lem:differential equation Dm} Let ${R}_{1},\,{R}_{2},\,R\in\mathcal{B}_{m}^{+}\subseteq\mathcal{L}_{k}$,
$1\leq m\leq k$, $k\in \mathbb{N}_{\geq 1}$, such that 
\[
\mathrm{ord}\big(z\boldsymbol{\ell}_{1}\cdots\boldsymbol{\ell}_{m-1}{R}_{1}\big)>\mathbf{1}_{k+1},\ \mathrm{ord}\big(z\boldsymbol{\ell}_{1}\cdots\boldsymbol{\ell}_{m-1}{R}_{2}\big)>\mathbf{1}_{k+1},
\]
and $\mathrm{ord}_{\boldsymbol{\ell}_{m}}(R)\geq2$. Let $h\in z^{2}\mathcal{B}_{\geq m+1}^{+}\left[\left[z\right]\right]$,
\emph{i.e.} $h$ is a formal series in $z$ with coefficients in $\mathcal{B}_{\ge m+1}^{+}$
such that $h\left(0\right)=h'\left(0\right)=0$. Then the differential
equation 
\begin{equation}
D_{m}(Q)+R_{1}\cdot Q+R_{2}\cdot h(Q)=R\label{eq:equadiff Dm}
\end{equation}
admits a unique solution in $\mathcal{B}_{m}^{+}$.
\end{lem}

\begin{proof}
Let ${h}=\sum_{n\geq2}{H}_{n}z^{n}$, with ${H}_{n}\in\mathcal{B}_{\geq m+1}^{+}$,
$n\geq2$. Define the operators $\mathcal{S}:\mathcal{B}_{m}^{+}\to\mathcal{B}_{m}^{+}$
and $\mathcal{T}:\mathcal{B}_{m}^{+}\to\mathcal{B}_{m}^{+}$ by 
\begin{align}
 & \mathcal{S}(Q):={R}-{R}_{2}\cdot\sum_{n\geq2}H_{n}\cdot Q^{n},\nonumber \\
 & \mathcal{T}(Q):=D_{m}(Q)+{R}_{1}\cdot Q,\ \ Q\in\mathcal{B}_{m}^{+}.\label{eq:T S equadiff Dm}
\end{align}
Note that $\mathcal{T}$ is a linear operator, while $\mathcal{S}$
is not. Now \eqref{eq:equadiff Dm} is equivalent to the fixed-point
equation 
\begin{equation}
\mathcal{S}(Q)=\mathcal{T}(Q),\ \ Q\in\mathcal{B}_{m}^{+}.\label{eq:fixed point equadiff Dm}
\end{equation}
By \eqref{eq:T S equadiff Dm} it is easy to see that $\mathcal{S}$
is a $\frac{1}{4}$-contraction and $\mathcal{T}$ is a $\frac{1}{2}$-homothety
on the space $\big(\mathcal{B}_{m}^{+},d_{m}\big)$.

Let us now prove that $\mathcal{S}(\mathcal{B}_{m}^{+})\subseteq\mathcal{T}(\mathcal{B}_{m}^{+})$.
Let ${M}\in\mathcal{S}(\mathcal{B}_{m}^{+})$. Note that, since $\mathrm{ord}_{\boldsymbol{\ell_{m}}}(R)\geq2$,
${M}\in\mathcal{S}(\mathcal{B}_{m}^{+})$ implies that $\mathrm{ord}_{\boldsymbol{\ell}_{m}}({M})\geq2$.
We prove now that ${M}\in\mathcal{T}(\mathcal{B}_{m}^{+}).$ That
is, by \eqref{eq:T S equadiff Dm}, that the equation 
\[
D_{m}(Q)+R_{1}\cdot Q=M
\]
has a solution $Q\in\mathcal{B}_{m}^{+}$. A direct computation shows
that this \emph{linear first-order} differential equation admits a
unique solution given by: 
\begin{equation}
Q:=\mathrm{exp}\left(-\int\frac{R_{1}}{\boldsymbol{\ell}_{m}^{2}}d\boldsymbol{\ell}_{m}\right)\cdot\int\frac{M}{\boldsymbol{\ell}_{m}^{2}}\cdot\mathrm{exp}\left(\int\frac{R_{1}}{\boldsymbol{\ell}_{m}^{2}}d\boldsymbol{\ell}_{m}\right)d\boldsymbol{\ell}_{m}.\label{eq:solution linear equadiff Dm}
\end{equation}
Here, $\mathrm{exp}$ denotes the formal exponential operator given
in Lemma~\ref{lem:exponential}. By Lemma~\ref{lem:primitive} (with
$\boldsymbol{\ell}_{m}$ playing the role of $z$ in Lemma \ref{lem:primitive}),
since $\mathrm{ord}\big(z\boldsymbol{\ell}_{1}\cdots\boldsymbol{\ell}_{m-1}{R}_{1}\big)>\mathbf{1}_{k+1}$,
it follows that $\int\frac{{R}_{1}}{\boldsymbol{\ell}_{m}^{2}}d\boldsymbol{\ell}_{m}\in\mathcal{B}_{\geq m}^{+}$.
For the same reasons, analyzing \eqref{eq:solution linear equadiff Dm} and since $\mathrm{ord}_{\boldsymbol{\ell}_{m}}\big(\frac{{M}}{\boldsymbol{\ell}_{m}^{2}}\big)\geq0$,
we obtain that $Q\in\mathcal{B}_{m}^{+}$.

Finally, Proposition~\ref{prop:krasno fixed point} provides a unique
solution in $\mathcal{B}_{m}^{+}$ to \eqref{eq:fixed point equadiff Dm}.
It is then the unique solution of \eqref{eq:equadiff Dm}.
\end{proof}
\begin{lem}
\label{lem:differential equation D1}Let ${R}_{1},{R}_{2},\ V\in\mathcal{B}_{1}^{+}\subseteq\mathcal{L}_{k}$,
$k\in\mathbb{N}_{\geq 1}$, such that $\mathrm{ord}(z{R}_{1}),\mathrm{ord}(z{R}_{2})>\mathbf{1}_{k+1}$
and $\mathrm{ord}(zV)>\mathbf{1}_{k+1}$. Let ${h}\in z^{2}\mathbb{R}\left[\left[z\right]\right]$.
Then the following equation 
\begin{equation}
D_{1}(Q)+R_{1}\cdot Q+R_{2}\cdot h(Q)=V\label{eq:equadiff D1}
\end{equation}
admits a unique solution ${Q}\in\mathcal{B}_{\geq1}^{+}$. 
\end{lem}

\begin{proof}
Since $\mathrm{ord}_{\boldsymbol{\ell}_{1}}(V)$ in \eqref{eq:equadiff D1}
is not necessarily bigger than $1$ (unlike in the hypotheses of
Lemma~\ref{lem:differential equation Dm}), we cannot directly apply
Lemma~\ref{lem:differential equation Dm} to prove the existence
of a solution. Therefore, we decompose adequately the right-hand side
of \eqref{eq:equadiff D1}, and then we successively apply Lemma~\ref{lem:differential equation Dm}
in the subspaces $\mathcal{B}_{m}^{+}$, $1\leq m\leq k$.

Since $\mathrm{ord}\big(zR_{1}\big)>\mathbf{1}_{k+1}$, $\mathrm{ord}\big(zR_{2}\big)>\mathbf{1}_{k+1}$
and $\mathrm{ord}\big(zV\big)>\mathbf{1}_{k+1}$, we have the unique
decompositions 
\begin{align*}
{V} & =\boldsymbol{\ell}_{1}\cdots\boldsymbol{\ell}_{k}{V}_{k}+\cdots+\boldsymbol{\ell}_{1}{V}_{1},\\
{R}_{1} & =\boldsymbol{\ell}_{1}\cdots\boldsymbol{\ell}_{k}{R}_{1,k}+\cdots+\boldsymbol{\ell}_{1}{R}_{1,1},\\
{R}_{2} & =\boldsymbol{\ell}_{1}\cdots\boldsymbol{\ell}_{k}{R}_{2,k}+\cdots+\boldsymbol{\ell}_{1}{R}_{2,1},
\end{align*}
where $V_{i},{R}_{1,i},{R}_{2,i}\in\mathcal{B}_{i}^{+}\subseteq\mathcal{L}_{k}$,
$i=1,\ldots,k$. We proceed inductively. \\

\emph{Step 1.} We assume, without loss of generality, that $V_{k}\neq0$
(simply consider the lowest $m$,\ $1\leq m<k$, such that $V_{m}\neq0$).
By Lemma \ref{lem:differential equation Dm} and \eqref{eq:lien D1 Dm},
the nonlinear differential equation 
\begin{align}
D_{1}(Q_{k})+\boldsymbol{\ell}_{1}\cdots\boldsymbol{\ell}_{k-1}\boldsymbol{\ell}_{k}\left({R}_{1,k}\cdot Q_{k}+{R}_{2,k}\cdot{h}(Q_{k})\right) & =\boldsymbol{\ell}_{1}\cdots\boldsymbol{\ell}_{k-1}\boldsymbol{\ell}_{k}V_{k},\label{Eq13}\\
\boldsymbol{\ell}_{1}\cdots\boldsymbol{\ell}_{k-1}D_{k}(Q_{k})+\boldsymbol{\ell}_{1}\cdots\boldsymbol{\ell}_{k-1}\boldsymbol{\ell}_{k}\left({R}_{1,k}\cdot Q_{k}+{R}_{2,k}\cdot{h}(Q_{k})\right) & =\boldsymbol{\ell}_{1}\cdots\boldsymbol{\ell}_{k-1}\boldsymbol{\ell}_{k}V_{k},\nonumber \\
D_{k}(Q_{k})+\boldsymbol{\ell}_{k}{R}_{1,k}\cdot Q_{k}+\boldsymbol{\ell}_{k}{R}_{2,k}\cdot{h}(Q_{k}) & =\boldsymbol{\ell}_{k}V_{k}.\nonumber 
\end{align}
admits a unique solution $Q_{k}\in\mathcal{B}_{k}^{+}$.\\

\emph{Step 2.} Let $Q_{k}\in\mathcal{B}_{k}^{+}$ be the solution
of \eqref{Eq13} as above. In the next step, using \eqref{eq:lien D1 Dm},
consider the nonlinear differential equation {\small{}
\begin{align}
 & D_{1}\big(Q_{k}+Q_{k-1}\big)+\boldsymbol{\ell}_{1}\cdots\boldsymbol{\ell}_{k-2}\left(\boldsymbol{\ell}_{k-1}\boldsymbol{\ell}_{k}R_{1,k}+\boldsymbol{\ell}_{k-1}R_{1,k-1}\right)\cdot\left(Q_{k}+Q_{k-1}\right)+\nonumber\\
 & \qquad\qquad\qquad\qquad\qquad+\boldsymbol{\ell}_{1}\cdots\boldsymbol{\ell}_{k-2}\left(\boldsymbol{\ell}_{k-1}\boldsymbol{\ell}_{k}R_{2,k}+\boldsymbol{\ell}_{k-1}R_{2,k-1}\right)\cdot{h}\left(Q_{k}+Q_{k-1}\right)=\nonumber \\
 & \qquad\qquad\qquad\qquad\qquad=\boldsymbol{\ell}_{1}\cdots\boldsymbol{\ell}_{k-2}\left(\boldsymbol{\ell}_{k-1}\boldsymbol{\ell}_{k}V_{k}+\boldsymbol{\ell}_{k-1}V_{k-1}\right),\label{eq:dva-1} \\
 & D_{1}(Q_{k-1})\!\!+\!\boldsymbol{\ell}_{1}\cdots\boldsymbol{\ell}_{k-2}\thinspace\!\!\left(\!\boldsymbol{\ell}_{k-1}\boldsymbol{\ell}_{k}R_{1,k}+\boldsymbol{\ell}_{k-1}R_{1,k-1}\!+\!\!\big(\boldsymbol{\ell}_{k-1}\boldsymbol{\ell}_{k}R_{2,k}+\boldsymbol{\ell}_{k-1}R_{2,k-1}\big)\cdot h'(Q_{k})\right)\!\!\cdot\!Q_{k-1}+\nonumber\\
 & \qquad\qquad\qquad\qquad\qquad+\boldsymbol{\ell}_{1}\cdots\boldsymbol{\ell}_{k-2}\left(\boldsymbol{\ell}_{k-1}\boldsymbol{\ell}_{k}R_{2,k}+\!\!\boldsymbol{\ell}_{k-1}R_{2,k-1}\right)\cdot\sum_{i\geq2}\frac{h^{(i)}(Q_{k})}{i!}Q_{k-1}^{i}=\nonumber \\
 & \qquad\qquad\qquad\qquad\qquad=\boldsymbol{\ell}_{1}\cdots\boldsymbol{\ell}_{k-2}\left(\boldsymbol{\ell}_{k-1}V_{k-1}-\boldsymbol{\ell}_{k-1}R_{1,k-1}Q_{k}-h(Q_{k})\boldsymbol{\ell}_{k-1}R_{2,k-1}\right),\label{eq:tri}\\
 & D_{k-1}(Q_{k-1})+\!\!\left(\boldsymbol{\ell}_{k-1}\boldsymbol{\ell}_{k}{R}_{1,k}+\boldsymbol{\ell}_{k-1}{R}_{1,k-1}+\!\!\big(\boldsymbol{\ell}_{k-1}\boldsymbol{\ell}_{k}{R}_{2,k}+\boldsymbol{\ell}_{k-1}{R}_{2,k-1}\big)\cdot{h}'({Q}_{k})\right)\cdot{Q}_{k-1}+\nonumber\\
 & \qquad\qquad\qquad\qquad\qquad+\left(\boldsymbol{\ell}_{k-1}\boldsymbol{\ell}_{k}{R}_{2,k}+\boldsymbol{\ell}_{k-1}{R}_{2,k-1}\right)\cdot\sum_{i\geq2}\frac{{h}^{(i)}({Q}_{k})}{i!}{Q}_{k-1}^{i}=\nonumber \\
 & \qquad\qquad\qquad\qquad\qquad=\boldsymbol{\ell}_{k-1}{V}_{k-1}-\boldsymbol{\ell}_{k-1}{R}_{1,k-1}{Q}_{k}-{h}({Q}_{k})\boldsymbol{\ell}_{k-1}{R}_{2,k-1}.\label{eq:cetiri-1}
\end{align}
}Equation \eqref{eq:tri} is obtained from \eqref{eq:dva-1} by using
\eqref{Eq13} to eliminate $D_{1}\left(Q_{k}\right)$, and by Taylor
expansion of $h$. Notice that the order $\mathrm{ord}_{\boldsymbol{\ell}_{k-1}}$
of the right-hand side of \eqref{eq:cetiri-1} is strictly bigger
than or equal to $2$. Therefore, by Lemma \ref{lem:differential equation Dm},
Equation \eqref{eq:cetiri-1} admits a unique solution $Q_{k-1}\in\mathcal{B}_{k-1}^{+}$.\\

\emph{Step 3.} Proceeding inductively in $k$ steps, we prove that
\eqref{eq:equadiff D1} admits a unique solution $Q:=Q_{1}+\cdots+Q_{k}\in\mathcal{B}_{\geq1}^{+}$.
\end{proof}

\subsubsection{\label{subsec:end case b}End of the proof of case (b).}
\begin{proof}[Proof of Lemma \ref{lem:hypotheses T0 S0}]
\,

$(1)$ From the definition of $\mathcal{T}_{0}$ in \eqref{eq:S0 and T0},
we deduce that
\[
\mathrm{ord}\big(\mathcal{T}_{0}({H})\big)=\mathrm{ord}\big(D_{1}({H})\big),\ {H}\in\mathcal{B}_{\geq1}^{+}.
\]
In particular, 
\[
\mathrm{ord}_{\boldsymbol{\ell}_{1}}\big(\mathcal{T}_{0}({H})\big)=\mathrm{ord}_{\boldsymbol{\ell}_{1}}\big(D_{1}({H})\big)=\mathrm{ord}_{\boldsymbol{\ell}_{1}}({H})+1.
\]
Therefore, $\mathcal{T}_{0}$ is a $\frac{1}{2}$-homothety.\\

$(2)$ Part 2 follows directly from the definition of the operator
$\mathcal{S}_{0}$ in \eqref{eq:S0 and T0}. \\

$(3)$ Let us prove that $\mathcal{S}_{0}(\mathcal{B}_{\geq1}^{+})\subseteq\mathcal{T}_{0}(\mathcal{B}_{\geq1}^{+})$.
Let ${M}\in\mathcal{S}_{0}(\mathcal{B}_{\geq1}^{+})$. By \eqref{eq:S0 and T0},
since $\mathrm{ord}(z{R})>\mathbf{1}_{k+1}$, it follows that $\mathrm{ord}(z{M})>\mathbf{1}_{k+1}$.
We prove that ${M}\in\mathcal{T}_{0}(\mathcal{B}_{\geq1}^{+})$. Indeed,
dividing both sides of
\begin{equation}
\mathcal{T}_{0}({H})={M}\label{eq:to}
\end{equation}
by $-\lambda\log\lambda-(1+\log\lambda)R_{0}-\lambda\sum_{i\geq2}\frac{(-1)^{i}}{i(i-1)}\left(\frac{R_{0}}{\lambda}\right)^{i}$
and applying Lemma~\ref{lem:differential equation D1}, we obtain
that there exists a ${H}\in\mathcal{B}_{\geq1}^{+}$ such that \eqref{eq:to}
holds.

Note that in the particular case where $D_{1}\left(R_{0}\right)=R$, there is a cancellation of
the linear part of $\mathcal{T}_{0}$ in \eqref{eq:S0 and T0}, and the proof does not go through. However, in this case,
we get simpler equations and can prove similarly the same result.
\end{proof}
\begin{proof}[Proof of case (b)]
Note that the prenormalizing transformation ${\varphi}_{0}\in\mathfrak{L}^{0}$
satisfying \eqref{eq:equation conjugacy phi0} is necessarily of the
form 
\[
\varphi_{0}(z)=z+zH+\mathrm{h.o.b.}(z),\ H\in\mathcal{B}_{\geq1}^{+}\subseteq\mathcal{L}_{m},
\]
for some $m\in\mathbb{N}_{\geq 1}$.

By Lemma~\ref{lem:fixed point T0 S0}, to prove the existence of
a prenormalization ${\varphi}_{0}$ it is enough to prove the existence
of a solution $H\in\mathcal{B}_{\geq1}^{+}$ to the fixed point equation
\eqref{eq:fixed point T0 and S0}. By Lemma~\ref{lem:hypotheses T0 S0},
equation \eqref{eq:fixed point T0 and S0} satisfies all the assumptions
of Proposition~\ref{prop:krasno fixed point} on spaces $\mathcal{B}_{\geq1}^{+}\subseteq\mathcal{L}_{m}$,
for every $m\geq k$. Therefore, there exists a unique solution $H\in\mathcal{B}_{\geq1}^{+}\subseteq\mathcal{L}_{m}$,
$m\geq k$, to equation \eqref{eq:fixed point T0 and S0}. By Lemma~\ref{lem:fixed point T0 S0},
${\varphi}_{0}(z)=z+zH+\mathrm{h.o.b.}(z)$ is a prenormalization
satisfying \eqref{eq:equation conjugacy phi0}, which is unique in
$\mathfrak{L}^{0}$ up to $\mathrm{h.o.b.}(z)$. Moreover, $\varphi_{0}$
belongs to the smallest $\mathcal{L}_{k}$ such that $f\in\mathcal{L}_{k}$.

Now we have that 
\[
\varphi_{0}\circ f\circ\varphi_{0}^{-1}=f_{0}+g,\ \text{for some }g\in\mathcal{L}_{k}\text{ such that}\ \mathrm{ord}_{z}({g})>1.
\]
Hence we can apply case (a) to reduce ${f}_{0}+{g}$ to the normal
form ${f}_{0}$. By the proof of case $(a)$, we know that there exists
a unique $\varphi_{1}\in\mathfrak{L}^{0}$, such that 
\[
\varphi_{1}\circ\varphi_{0}\circ f\circ\varphi_{0}^{-1}\circ\varphi_{1}^{-1}=f_{0}.
\]
Moreover, ${\varphi}_{1}\in\mathcal{L}_{k}$.

Now, ${\varphi}:={\varphi}_{1}\circ{\varphi}_{0}\in\mathcal{L}_{k}\cap\mathfrak{L}^{0}$
is the normalizing change of variables in case (b), \emph{i.e.} 
\begin{equation}
\varphi\circ f\circ\varphi^{-1}=f_{0}.\label{eq:sve}
\end{equation}

It remains to prove the uniqueness of the normalizing change of variables
${\varphi}\in\mathfrak{L}^{0}$. Suppose that there exists another
change of variables ${\psi}\neq{\varphi}$, ${\psi}\in\mathfrak{L}^{0}$,
satisfying \eqref{eq:sve}. Then ${\psi}\in\mathcal{L}_{m}$, for
some $m\geq k$. Let us decompose ${\psi}$ as ${\psi}={\psi}_{1}\circ{\psi}_{0}$,
where ${\psi}_{0}$ is of the form ${\psi}_{0}(z)=z+zV+\mathrm{h.o.b.}(z)$,
$V\in\mathcal{B}_{\geq1}^{+}$, and $\mathrm{ord}_{z}({\psi}_{1}-\mathrm{id})>1$.
It is easy to see that such a $V$ is unique. Now, 
\begin{align*}
{\psi}_{0}\circ{f}\circ{\psi}_{0}^{-1}={\psi}_{1}^{-1}\circ{f}_{0}\circ{\psi}_{1}={f}_{0}+{g}_{1},\\
{\varphi}_{0}\circ{f}\circ{\varphi}_{0}^{-1}={\varphi}_{1}^{-1}\circ{f}_{0}\circ{\varphi}_{1}={f}_{0}+{g}_{2},
\end{align*}
where $\mathrm{ord}_{z}(g_{1}),\ \mathrm{ord}_{z}(g_{2})>1$, are
both prenormalization equations for ${f}$. By the proof of case (b),
$V=H$.

Now let $\tilde{\psi}_{1}:={\psi}\circ(z+zV)^{-1}\in\mathfrak{L}^{0}$
and $\tilde{\varphi}_{1}:={\varphi}\circ(z+zV)^{-1}\in\mathfrak{L}^{0}$, where $(z+zV)^{-1}$ is the compositional inverse of the transseries $z+zV$.
Set ${f}_{1}:=(z+zV)\circ{f}\circ(z+zV)^{-1}={f}_{0}+{h}$, $\mathrm{ord}_{z}({h})>1$.
Then 
\[
\tilde{\psi}_{1}\circ{f}_{1}\circ\tilde{\psi}_{1}^{-1}={f}_{0},\ \tilde{\varphi}_{1}\circ{f}_{1}\circ\tilde{\varphi}_{1}^{-1}={f}_{0}.
\]
By the uniqueness in the proof of case (a), $\tilde{\psi}_{1}=\tilde{\varphi}_{1}$.
Therefore, ${\psi}={\varphi}$. 
\end{proof}

\subsection{Proof of the minimality of the normal forms.}

\label{subsec:minimality normal forms}Since $\mathfrak{L}^{0}$ is
a group for composition, it is sufficient to prove that, by changes
of variables in $\mathfrak{L}^{0}$, we cannot further reduce ${f}_{0}\in\mathcal{L}_{k}$,
$k\in\mathbb{N}$, given in \eqref{eq:normal form}
to a strictly \emph{shorter} normal form with possibly different coefficients.
Let ${r}(z)=\eta z+\mathrm{h.o.t.}\in\mathfrak{L}^{H}$, with $\eta>0$. The existence of a transseries $\varphi \in \mathfrak{L}^{0}$ such that $\varphi \circ f_{0}\circ \varphi ^{-1}=r$ implies that $\eta =\lambda $.

By a direct generalization of Lemma~\ref{lem:fixed point T0 S0}
with $f_{0}$ instead of $f$ and $r$ instead of $f_{0}$, the existence
of a transseries ${\varphi}\in\mathfrak{L}^{0}$ which satisfies 
\begin{equation*}
\varphi\circ f_{0}\circ\varphi^{-1}=r,
\end{equation*}
is equivalent, by setting ${\varphi}=z+zH$, with $\mathrm{ord}(zH)>(1,\mathbf{0}_{k}),$
to the existence of $H\in\mathcal{B}_{\geq1}^{+}\subseteq\mathcal{L}_{k}$
which satisfies 
\begin{equation}
\mathcal{S}_{0}(H)=\mathcal{T}_{0}(H),\label{eq:uh}
\end{equation}
where $\mathcal{S}_{0}$ and $\mathcal{T}_{0}$ are as in \eqref{eq:S0 and T0}.
Notice that $r$ and $\varphi$ do not necessary belong to $\mathcal{L}_{k}$,
but to $\mathcal{L}_{m}$ for some $m\ge k$. Hence in this proof
the orders actually belong to $\mathbb{R}_{\ge0}\times\mathbb{Z}^{m}$,
via the canonical inclusion $\mathbb{R}_{\ge0}\times\mathbb{Z}^{k}\subseteq\mathbb{R}_{\ge0}\times\mathbb{Z}^{m}$.
Let ${g}:=f_{0}-r$. We prove that $\mathrm{ord}({g})>\mathbf{1}_{k+1}$ if such an $H$ exists. Suppose now that $\mathrm{ord}({f}_{0}-{r})\leq\mathbf{1}_{k+1}$
and that \eqref{eq:uh} can be solved with some $H\in\mathcal{B}_{\geq1}^{+}$,
$\mathrm{ord}({H})>\mathbf{0}_{k+1}$. Let $R\in \mathcal B_{\geq 1}^+$ such that $g=zR+\mathrm{h.o.b}(z)$. Let $R_0\in \mathcal B_{\geq 1}^+$ be such that $r=z+zR_0$. Then $\mathrm{ord}(R)\leq(0,\mathbf 1_k)$. Note that $R$ and $R_0$ differ from those originally defined in \eqref{eq:S0 and T0}. Analyzing the leading monomials
in the expressions \eqref{eq:S0 and T0} for $\mathcal{T}_{0}(H)$
and $\mathcal{S}_{0}(H)$, due to the fact that $\mathrm{ord}({g})\leq\mathbf{1}_{k+1}$,
we see that
\begin{align}
 & \mathrm{ord}(\mathcal{S}_{0}(H))=\mathrm{ord}({R}),\nonumber \\
 & \mathrm{ord}(\mathcal{T}_{0}(H))=\min\{\mathrm{ord}(H)+\mathrm{ord}({R}),\mathrm{ord}(D_1(H))\}.\label{eq:ord T0}
\end{align}
Indeed, $D_{1}$ increases the order by at most $(0,\mathbf{1}_{k})$, and the same holds for multiplication by $R$. Since $\mathrm{ord}(H)>\mathbf{0}_{k+1}$, $\mathrm{ord}(D_1 (H))>(0,\mathbf 1_{k})$ and since $\mathrm{ord}(R)\leq(0,\mathbf 1_k)$, 
\eqref{eq:ord T0} leads to a contradiction with \eqref{eq:uh}.

Since $\mathrm{ord}(g)=\mathrm{ord}(f_{0}-r)>\boldsymbol{1}_{k+1}$ and the order of every term in $f_{0}$ is at most $\boldsymbol{1}_{k+1}$, the minimality of $f_{0}$ follows.
\begin{rem}[A normalization algorithm]
\label{rem:constr} We finish this section by the following comment
on the proof of the first two parts of the Main Theorem. While the
proof of the similar result in \cite{mrrz1} was by nature purely
transfinite, as it consisted in eliminating, one after the other,
all the terms of the initial series $f$, the proof presented here
is more ``constructive'': it provides an algorithm for the construction
of the normalization $\varphi$ of a hyperbolic transseries $f$,
given as a composition 
\[
\varphi=\varphi_{1}\circ\varphi_{0}
\]
of the limits of two Picard sequences, as follows.

Suppose ${f}=\lambda z+{g}$, $0<\lambda<1$, where ${f}\in\mathcal{L}_{k}^{H}$ and
$\mathrm{ord}(g)>(1,\boldsymbol{0}_{k})$. If $\mathrm{ord}_{z}({g})=1$, we first
construct the prenormalizing transformation ${\varphi}_{0}=z+zH$,
as described in Subsection~\ref{subsec:proof case b}, where $H\in\mathcal{B}_{\geq1}^{+}$
is given as 
\begin{equation}
H:=\lim_{n\to\infty}(\mathcal{T}_{0}^{-1}\circ\mathcal{S}_{0})^{\circ n}(H_{0}),\label{eq:Pic}
\end{equation}
where $\mathcal{S}_{0},\ \mathcal{T}_{0}$ are defined in \eqref{eq:S0 and T0}
in Lemma~\ref{lem:fixed point T0 S0}. We can take any $H_{0}\in\mathcal{B}_{\geq1}^{+}$
as an initial condition. The Picard sequence \eqref{eq:Pic} converges with respect to
$d_{1}$ in $\mathcal{B}_{\geq1}^{+}$, by the Banach
fixed point theorem and Lemmas~\ref{lem:fixed point T0 S0} and \ref{lem:hypotheses T0 S0}.

Now, ${f}_{1}:={\varphi}_{0}\circ{f}\circ{\varphi}_{0}^{-1}={f}_{0}+{g}_{1}$,
where $\beta:=\mathrm{ord}_{z}(g_{1})>1$, is \emph{prenormalized}.
By Subsection~\ref{subsec:proof case a}, the normalizing change
of variables ${\varphi}_{1}$ is of the form ${\varphi}_{1}=\mathrm{id}+{h}$,
where ${h}\in\mathcal{L}_{k}$, $\mathrm{ord}_{z}({h})>1$, is given
by the limit of the following Picard sequence:
\[
{h}:=\lim_{n\to\infty}(\mathcal{T}_{{f}_{1}}^{-1}\circ\mathcal{S}_{{f}_{1}})^{\circ n}({h}_{0}).
\]
The above sequence converges in $\mathcal{L}_{m}^{\geq\beta}$ in
the valuation topology, for any ${h}_{0}\in\mathcal{L}_{m}^{\geq\beta}$,
where $m\geq k$, by Lemma~\ref{lem:properties Sf Tf } and Proposition
\ref{prop:krasno fixed point}. Here, $\mathcal{S}_{{f}_{1}}$ and
$\mathcal{T}_{{f}_{1}}$ are as in \eqref{eq:operators S and T}.

It remains to prove the third part of the statement of the Main Theorem,
which will be done in the next subsection.
\end{rem}

\subsection{Proof of the convergence of the generalized Koenigs sequence}

\label{subsec:convergence koenigs}

The sequence defined in \eqref{eq:gK} below is a generalization to
logarithmic transseries of the standard Koenigs sequence $\Big(\frac{{f}^{\circ n}}{\lambda^{n}}\Big)_{n}$(see
e.g.\,\cite{carleson-gamelin:complex-dynamics}). This sequence is
not the Picard sequence which was obtained in Subsections~\ref{subsec:proof case a}
and \ref{subsec:proof case b}, and which converges to the normalization
$\varphi$. Indeed, in the case of a prenormalized ${f}\in\mathcal{L}_{k}^{H}$,
the Picard sequence converges to the normalization ${\varphi}$ in the
valuation topology (see Remark~\ref{rem:constr} for more details).
This is not the case for the Koenigs sequence, which, however, converges
in the weak topology.
\begin{lem}[Convergence of the generalized Koenigs sequence]
\label{lem:koenigs} Let $k\in\mathbb{N}$ and ${f}\in\mathcal{L}_{k}^{H}$
be hyperbolic, and let ${f}_{0}$ be its formal normal form from \eqref{eq:normal form}.
For a parabolic initial condition ${h}\in\mathfrak{L}^{0}$, the \emph{generalized
Koenigs sequence} 
\begin{eqnarray}
\mathcal{P}_{{f}}^{\circ n}(h):=\Big(f_{0}^{\circ\left(-n\right)}\circ h\circ f^{\circ n}\Big)_{n}\label{eq:gK}
\end{eqnarray}
converges to the normalization $\varphi\in\mathfrak{L}^{0}$ in the
weak topology, as $n\to\infty$, if and only if $\mathrm{Lb}_{z}({h})=\mathrm{Lb}_{z}({\varphi})$. 
\end{lem}

In this statement, the convergence in the weak topology is seen in
the class $\mathcal{L}_{r}$, where $r:=\max\{k,m\},$ with $f\in\mathcal{L}_{k}$
and ${h}\in\mathcal{L}_{m}$. In the proof, we will need the following
general result.
\begin{lem}[Sequential continuity]
\label{lemma:sequential continuity} Let $\left({g}_{n}\right)_{n\in\mathbb{N}}$
be a sequence in $\mathcal{L}_{r}$, $r\in\mathbb{N}$, such that
$\mathrm{Supp}\left({g}_{n}\right),\,n\in\mathbb{N},$ are contained
in a common well-ordered subset $W$ of $\mathbb{R}_{\geq0}\times\mathbb{Z}^{r}$ such that $\min W>\boldsymbol{0}_{r+1}$.
Let ${\varphi}\in\mathcal{L}_{r}^{0}$. Then 
\[
{g}_{n}\to{g}_{0},\ \text{as }n\to\infty,
\]
if and only if 
\[
{g}_{n}\circ{\varphi}\to{g}_{0}\circ{\varphi},\ \text{as }n\to\infty,
\]
both in the weak topology on $\mathcal{L}_{r}$.
\end{lem}

\begin{proof}
Suppose that $g_{n}\to g_{0}$, as $n\to\infty$, with respect to
the weak topology. Put $\varphi:=\mathrm{id}+\varphi_{1},$ where $\varphi_{1}\in\mathcal{L}_{r}$,
$\mathrm{ord}(\varphi_{1})>(1,\mathbf{0}_{r})$. By Taylor's expansion (Proposition \ref{prop:formal Taylor}), we have: 
\[
g_{n}\circ\varphi=g_{n}\big(\mathrm{id}+{\varphi}_{1}\big)=g_{n}+\sum_{i\geq1}\frac{g_{n}^{(i)}}{i!}{\varphi}_{1}^{i},\text{ for}\ n\in\mathbb{N}.
\]
It follows easily that, for $n\in\mathbb{N}$, the supports of $g_{n}\circ\varphi$
are contained in the semigroup $G$ generated by $W$ and
\[
\left\lbrace \left(\alpha-1,\mathbf m\right),(0,1,0,\ldots,0)_{r+1},\ldots,(0,0,\ldots,0,1)_{r+1}\right\rbrace ,
\]
for every $\left(\alpha,\mathbf m\right)\in\mathrm{Supp}({\varphi}_{1})$.
Let $\boldsymbol{w}\in G$. By Neumann's Lemma, there exists $k_{\boldsymbol{w}}\in\mathbb{N}$
and a linear polynomial $P_{\boldsymbol{w}}$ in $k_{\boldsymbol{w}}$
variables (the coefficients of which depend on ${\varphi}_{1}$ but
not on $n\in\mathbb{N}$), such that 
\begin{equation}
\left[{g}_{n}\circ{\varphi}\right]_{\boldsymbol{w}}=P_{\boldsymbol{w}}\left(\left[g_{n}\right]_{\boldsymbol{w}_{1}},\ldots,\left[g_{n}\right]_{\boldsymbol{w}_{k_{\boldsymbol{w}}}}\right),\ \text{for }n\in\mathbb{N}.\label{Eq15}
\end{equation}
Here, \textbf{$\boldsymbol{w}_{1}$},\ldots ,$\boldsymbol{w}_{k_{\boldsymbol{w}}}$
are finitely many elements of $G$, independent of $n\in\mathbb{N}$.
By continuity of polynomial functions, we have: 
\begin{equation}
P_{\boldsymbol w}\left(\left[g_{n}\right]_{\boldsymbol{w}_{1}},\ldots,\left[g_{n}\right]_{\boldsymbol{w}_{k_{\boldsymbol{w}}}}\right)\underset{n\rightarrow\infty}{\longrightarrow}P_{\boldsymbol w}\left(\left[g_{0}\right]_{\boldsymbol{w}_{1}},\ldots,\left[g_{0}\right]_{\boldsymbol{w}_{k_{\boldsymbol{w}}}}\right).\label{eq:tij}
\end{equation}
Thus, using \eqref{Eq15} and \eqref{eq:tij}, we obtain, for every
$\boldsymbol{w}\in G$: 
\[
\left[g_{n}\circ\varphi\right]_{\boldsymbol{w}}\underset{n\rightarrow\infty}{\longrightarrow}\left[g_{0}\circ\varphi\right]_{\boldsymbol{w}}.
\]

Now we prove the converse. Suppose that ${g}_{n}\circ{\varphi}\to{g}_{0}\circ{\varphi}$,
as $n\to\infty$, in the weak topology. Composing by $\varphi^{-1}$
and using the same arguments as above, we deduce that $g_{n}\to g_{0}$,
as $n\to\infty$, in the weak topology.
\end{proof}
\begin{proof}[Proof of Lemma~\ref{lem:koenigs}]
Notice that, for ${h},\,{\varphi}\in\mathfrak{L}^{0}$, $\mathrm{Lb}_{z}({h})=\mathrm{Lb}_{z}({\varphi})$
if and only if $\mathrm{ord}_{z}\big({h}\circ{\varphi}^{-1}-\mathrm{id}\big)>1$.

($\Leftarrow$) Let ${h}\in\mathfrak{L}^{0}$ and let ${\varphi}\in\mathcal{L}_{k}^{0}$
be the normalization of ${f}$. Let $m\in\mathbb{N}$ be the smallest
integer such that ${h}\in\mathcal{L}_{m}^{0}$ and $m\geq k$. Let
$\mathrm{ord}_{z}\big({h}\circ{\varphi}^{-1}-\mathrm{id}\big)>1$.

Since ${\varphi}\circ{f}^{\circ n}\circ{\varphi}^{-1}={f}_{0}^{\circ n}$
for every $n\in\mathbb{N}$, it is easy to show that
\[
\mathcal{P}_{f}^{\circ n}(h)\circ\varphi^{-1}=f_{0}^{\circ(-n)}\circ(h\circ\varphi^{-1})\circ f_{0}^{\circ n},\ \forall n\in\mathbb{N}.
\]
Note that $\mathcal{P}_{{f}}^{\circ n}({h})\circ{\varphi}^{-1}\in\mathcal{L}_{m}^{0}$.
Therefore, by Lemma~\ref{lemma:sequential continuity} applied to
$g_{n}:=\mathcal{P}_{f}^{\circ n}(h)\circ\varphi^{-1}$ (which have
support in the common well-ordered set $\widetilde{\mathcal{H}} \cup \left\lbrace (1,\boldsymbol{0}_{m})\right\rbrace $, for $\widetilde{\mathcal{H}}$ given below by \eqref{eq:definition H tilda}),
in order to prove the convergence of \eqref{eq:gK} to the normalization $\varphi$, it is sufficient to prove the equivalent
statement 
\[
f_{0}^{\circ(-n)}\circ(h\circ\varphi^{-1})\circ f_{0}^{\circ n}\underset{n\rightarrow\infty}{\longrightarrow}\mathrm{id}
\]
in the weak topology in $\mathcal{L}_{m}^{0}$.

Let $h_{1}$ be defined by ${h}\circ{\varphi}^{-1}=\mathrm{id}+h_{1}$.
Then, by assumption, $\mathrm{ord}_{z}(h_{1})>1$. By Taylor's expansion (Proposition \ref{prop:formal Taylor}), it follows that 
\begin{equation}
f_{0}^{\circ(-n)}\circ(h\circ\varphi^{-1})\circ f_{0}^{\circ n}=f_{0}^{\circ(-n)}\circ(\mathrm{id}+h_{1})\circ f_{0}^{\circ n}=\mathrm{id}+\sum_{i\geq1}\frac{\big(f_{0}^{\circ(-n)}\big)^{(i)}(f_{0}^{\circ n})}{i!}h_{1}^{i}(f_{0}^{\circ n}).\label{eq:tayli}
\end{equation}
By \eqref{eq:tayli}, for $n\in\mathbb{N}$, the leading block of
$f_{0}^{\circ(-n)}\circ(h\circ\varphi^{-1})\circ f_{0}^{\circ n}$
is equal to $z$. Let us define the set $\widetilde{\mathcal{H}}\subseteq\mathbb{R}_{>1}\times\mathbb{Z}^{m}$
by
\begin{equation}
\widetilde{\mathcal{H}}:=\bigcup_{n\in\mathbb{N}}\mathrm{Supp}\big(f_{0}^{\circ(-n)}\circ(h\circ\varphi^{-1})\circ f_{0}^{\circ n}-\mathrm{id}\big).\label{eq:definition H tilda}
\end{equation}
It can be shown using \eqref{eq:tayli} and the fact that $f_0$ contains only block of order $1$ in $z$ that $\widetilde{\mathcal{H}}$
is well-ordered. It remains to prove that 
\[
\big[{f}_{0}^{\circ(-n)}\circ({h}\circ{\varphi}^{-1})\circ{f}_{0}^{\circ n}\big]_{\alpha,\mathbf{k}}\underset{n\rightarrow\infty}{\longrightarrow}0
\]
for every $(\alpha,\mathbf{k})\in\widetilde{\mathcal{H}}$. We prove
this by transfinite induction on $\widetilde{\mathcal{H}}$.

\emph{The base case}. Let $(\alpha_{0},\mathbf{k}_{0}):=\min\widetilde{\mathcal{H}}$,
$\mathbf{k}_{0}\in\mathbb{Z}^{m}$. By \eqref{eq:tayli}, since ${f}_{0}=\lambda\cdot\mathrm{id}+\mathrm{h.o.t.}$,
\begin{equation}
\mathrm{Lt}\big({f}_{0}^{\circ(-n)}\circ(\mathrm{id}+h_{1})\circ{f}_{0}^{\circ n}-\mathrm{id}\big)=\lambda^{n(\alpha-1)}\mathrm{Lt}(h_{1}),\ \text{for }n\in\mathbb{N},\label{eq:vso}
\end{equation}
where $\alpha:=\mathrm{ord}_{z}(h_{1})$. Therefore, by definition
of $\widetilde{\mathcal{H}}$, $\mathrm{ord}(h_{1})=(\alpha_{0},\mathbf{k}_{0}).$
Now, by \eqref{eq:vso}, 
\[
\big[f_{0}^{\circ(-n)}\circ(h\circ\varphi^{-1})\circ f_{0}^{\circ n}\big]_{\alpha_{0},\mathbf{k}_{0}}=\lambda^{n(\alpha_{0}-1)}\big[h_{1}\big]_{\alpha_{0},\mathbf{k}_{0}}\to0,\ \text{as }n\to\infty,
\]
since $\alpha_{0}-1>0$ and $0<\lambda<1$.

\emph{The induction step.} For simplicity, let us denote, for $(\gamma,\mathbf{r})\in\widetilde{\mathcal{H}}$
and $n\in\mathbb{N}$, 
\begin{align}
\begin{split}P_{\gamma,\mathbf{r}}^{n} & :=\big[{f}_{0}^{\circ(-n)}\circ(\mathrm{id}+h_{1})\circ{f}_{0}^{\circ n}\big]_{\gamma,\mathbf{r}}.
\end{split}
\label{eq:definition Pn and Rn}
\end{align}
Suppose that $(\beta,\mathbf{m})\in\widetilde{\mathcal{H}}$ and $(\beta,\mathbf{m})>(\alpha_{0},\mathbf{k}_{0})$
and that 
\begin{equation*}
P_{\gamma,\mathbf{r}}^{n}\to0,\ \text{as }n\to\infty,\ \forall(\gamma,\mathbf{r})\in\widetilde{\mathcal{H}},\text{ such that}\ (\gamma,\mathbf{r})<(\beta,\mathbf{m}).
\end{equation*}
We prove that $P_{\beta,\mathbf{m}}^{n}\to0,\ \text{as }n\to\infty.$

Using inductively Lemma~\ref{lem:okk} (the proof of which is just
after the current one), we obtain:
\begin{align}
\begin{split}P_{\beta,\mathbf{m}}^{n+1} & =A_{\beta,\mathbf{m}}\big(\{P_{\gamma,\mathbf{n}}^{n}:\ (\gamma,\mathbf{n})\in\widetilde{\mathcal{H}},\ (\gamma,\mathbf{n})<(\beta,\mathbf{m})\}\big)+\\
 & \quad+\lambda^{\beta-1}A_{\beta,\mathbf{m}}\big(\{P_{\gamma,\mathbf{n}}^{n-1}:\ (\gamma,\mathbf{n})\in\widetilde{\mathcal{H}},\ (\gamma,\mathbf{n})<(\beta,\mathbf{m})\}\big)+\ldots+\\
 & \quad\quad+\lambda^{n(\beta-1)}A_{\beta,\mathbf{m}}\big(\{P_{\gamma,\mathbf{n}}^{0}:\ (\gamma,\mathbf{n})\in\widetilde{\mathcal{H}},\ (\gamma,\mathbf{n})<(\beta,\mathbf{m})\}\big)+\lambda^{(n+1)(\beta-1)}P_{\beta,\mathbf{m}}^{0}\\
 & =\sum_{i=0}^{n}\lambda^{i(\beta-1)}A_{\beta,\mathbf{m}}\big(\{P_{\gamma,\mathbf{n}}^{n-i}:\ (\gamma,\mathbf{n})\in\widetilde{\mathcal{H}},\ (\gamma,\mathbf{n})<(\beta,\mathbf{m})\}\big)+\lambda^{(n+1)(\beta-1)}[h_{1}]_{\beta,\mathbf m}.
\end{split}
\label{eq:uuse}
\end{align}
Note that, for $(\beta,\mathbf{m})>(1,\mathbf{0}_{m})$, $P_{\beta,\mathbf{m}}^{0}=[h_{1}]_{\beta,\mathbf m}$.
Let $a_{\gamma,\mathbf n}\in\mathbb{R}$ be the nonzero coefficient of $P_{\gamma,\mathbf n}^n$ (it does not depend on $n\in\mathbb N$) in the polynomial $A_{\beta,\mathbf{m}}$, $(1,\mathbf{0}_{m})<(\gamma,\mathbf{n})<(\beta,\mathbf{m})$.
We prove that the sum 
\begin{equation}
a_{\gamma,\mathbf n}\sum_{i=0}^{n}\lambda^{i(\beta-1)}P_{\gamma,\mathbf{n}}^{n-i}\label{eq:sumii}
\end{equation}
converges to $0$, as $n\to\infty$. Then, since the first sum in
the last row of \eqref{eq:uuse} is a sum of finitely many sums of
type \eqref{eq:sumii}, it converges to $0$, as $n\to\infty$. Moreover,
since $0<|\lambda|<1$, $\lambda^{(\beta-1)(n+1)}\to0$, as $n\to\infty$.
Therefore, by \eqref{eq:uuse}, $P_{\beta,\mathbf{m}}^{n+1}\to0$,
as $n\to\infty$. This proves the induction step.

It remains to prove the convergence of \eqref{eq:sumii} to $0$ as
$n\rightarrow\infty$. We observe that \eqref{eq:sumii} is, up to
the factor $a_{\gamma,\mathbf n}$, the general term of the discrete convolution product
of the sequence $\left(\lambda^{i\left(\beta-1\right)}\right)_{i\in\mathbb{N}}$
and the sequence $\left(P_{\gamma,\mathbf{n}}^{i}\right)_{i\in\mathbb{N}}$.
The series $\sum_{i\in\mathbb{N}}\lambda^{i\left(\beta-1\right)}$
is absolutely convergent and $\sum_{i\in\mathbb{N}}\left|\lambda^{i\left(\beta-1\right)}\right|=\dfrac{1}{1-\left|\lambda\right|^{\beta-1}}$.
Hence, up to the multiplication by $1-\left|\lambda\right|^{\beta-1}$,
this convolution can be expressed as the product of the infinite vector
$\left(P_{\gamma,\mathbf{n}}^{i}\right)_{i\in\mathbb{N}}$ by an
infinite \emph{Toeplitz matrix}. It is well known (see for example
\cite[Section 2.16, Theorem 1]{goffman_pedrick:first_course_functional_analysis})
that such a product is a \emph{regular method of summability}, which
respects the limits of convergent sequences. Since, by hypothesis,
$P_{\gamma,\mathbf{n}}^{i}\underset{i\rightarrow\infty}{\rightarrow}0$,
it follows that \eqref{eq:sumii} tends to $0$ as $n\rightarrow\infty$.\\

$(\Rightarrow)$ Conversely, let ${h}\in\mathfrak{L}^{0}$ and let
$m\in\mathbb{N}$ be the minimal integer such that ${h},\ {f}\in\mathcal{L}_{m}$.
Suppose that
\begin{equation*}
f_{0}^{\circ(-n)}\circ(h\circ\varphi^{-1})\circ f_{0}^{\circ n}\underset{n\rightarrow\infty}{\longrightarrow}\mathrm{id},
\end{equation*}
in the weak topology in $\mathcal{L}_{m}^{0}$, and that $\mathrm{ord}_{z}\big({h}\circ{\varphi}^{-1}-\mathrm{id}\big)=1$.
Setting ${h}\circ{\varphi}^{-1}=z+z{R}+\mathrm{h.o.b.}(z)\,$, where
${R}\in\mathcal{B}_{\geq 1}^{+}$, $R\neq0$, we compute by Taylor's expansion (Proposition \ref{prop:formal Taylor}):
\begin{align}
\begin{split}{f}_{0}^{\circ(-n)} & \circ(z+z{R}+\mathrm{h.o.b.}(z))\circ{f}_{0}^{\circ n}=\mathrm{id}+\frac{{f}_{0}^{\circ n}\cdot{R}({f}_{0}^{\circ n})}{\frac{d}{dz}{f}_{0}^{\circ n}}+\cdots
\end{split}
\label{eq:spl}
\end{align}
Now, since $f_{0}^{\circ n}=\lambda^{n}z+\mathrm{h.o.t}$. and
since $\mathrm{Lt}\left(R\left(f_{0}^{\circ n}\right)\right)=\mathrm{Lt}\left(R\right)$,
by \eqref{eq:spl} we get that
\begin{equation}
\begin{split}f_{0} & ^{\circ(-n)}\circ(z+zR+\mathrm{h.o.b.}(z))\circ f_{0}^{\circ n}-\mathrm{id}=z\mathrm{Lt}(R)+\mathrm{h.o.t.}\end{split}
\label{eq:taji}
\end{equation}
The first term does not change with $n$, and therefore the right-hand
side of \eqref{eq:taji} does not converge to $0$ in the weak topology. 
\end{proof}
Finally, the next lemma is a technical result used in in the proof
of Lemma \ref{lem:koenigs}:
\begin{lem}
\label{lem:okk} Let $m\in\mathbb{N}$ and ${f}\in\mathcal{L}_{m}^{H}$
be hyperbolic. Let ${f}_{0}$ be its formal normal form and $\varphi\in\mathcal{L}_{m}^{0}$
the normalization of $f$. Let ${h}\in\mathcal{L}_{m}^{0}$ such that
$\mathrm{Lb}{}_{z}(h)=\mathrm{Lb}_{z}(\varphi)$ and let $\widetilde{\mathcal{H}}$
be as in \eqref{eq:definition H tilda}. Let $P_{\beta,\mathbf{m}}^{n}$,
for $(\beta,\mathbf{m})\in\widetilde{\mathcal{H}}$ and $n\in\mathbb{N}$,
be as in \eqref{eq:definition Pn and Rn}. Then
\begin{equation*}
P_{\beta,\mathbf{m}}^{n+1}=A_{\beta,\mathbf{m}}\big(\{P_{\gamma,\mathbf{n}}^{n}:\ (\gamma,\mathbf{n})\in\widetilde{\mathcal{H}},\ (\gamma,\mathbf{n})<(\beta,\mathbf{m})\}\big)+\lambda^{\beta-1}P_{\beta,\mathbf{m}}^{n},\ \text{for }n\in\mathbb{N}.
\end{equation*}
Here, for $(\beta,\mathbf{m})\in\widetilde{\mathcal{H}}$, the $A_{\beta,\mathbf{m}}$
are linear polynomials in the variables $\{P_{\gamma,\mathbf{n}}^{n}:\ (\gamma,\mathbf{n})\in\widetilde{\mathcal{H}},\ (\gamma,\mathbf{n})<(\beta,\mathbf{m})\}$,
whose coefficients are independent of $n$. The coefficients depend
only on ${f}$, $h$ and $\varphi$.
\end{lem}

In this statement, $m\in\mathbb{N}$ is the minimal integer such that
both ${f}$ and ${h}$ belong to $\mathcal{L}_{m}$.
\begin{proof}
\ Let $P_{\gamma,\mathbf{r}}^{n}$,
$(\gamma,\mathbf{r})\in\widetilde{\mathcal{H}}$, $n\in\mathbb{N}$,
be as in \eqref{eq:definition Pn and Rn}. We define, for $(\gamma,\mathbf{r})\in\widetilde{\mathcal{H}}$, $n\in\mathbb{N}$, 
\begin{align}\label{eq:Rg}
R_{\gamma,\mathbf{r}}^{n}:=\big[{f}_{0}^{\circ(-n+1)}\circ(\mathrm{id}+h_{1})\circ{f}_{0}^{\circ n}\big]_{\gamma,\mathbf{r}}.
\end{align}
First, we prove that 
\begin{equation}
R_{\beta,\mathbf{m}}^{n+1}=B_{\beta,\mathbf{m}}\big(\{P_{\beta,\mathbf{n}}^{n}:\ (\beta,\mathbf{n})\in\widetilde{\mathcal{H}},\ \mathbf{n}<\mathbf{m}\}\big)+\lambda^{\beta}\cdot P_{\beta,\mathbf{m}}^{n},\ (\beta,\mathbf{m})\in\widetilde{\mathcal{H}},\ n\in\mathbb{N},\label{eq:shape Rn}
\end{equation}
where $B_{\beta,\mathbf{m}}$, $(\beta,\mathbf{m})\in\widetilde{\mathcal{H}}$,
are linear polynomials in the variables $\{P_{\beta,\mathbf{n}}^{n}\in\widetilde{\mathcal{H}}:\,\mathbf{n}<\mathbf{m}\}$,
whose coefficients are independent of $n$. Indeed, for any $(\beta,\mathbf{m})\in\widetilde{\mathcal{H}}$,
$\mathbf{m}=(m_{1},\ldots,m_{m})\in\mathbb{Z}^{m}$, and $n\in\mathbb{N}$,
\[
\big(P_{\beta,\mathbf{m}}^{n}z^{\beta}\boldsymbol{\ell}_{1}^{m_{1}}\cdots\boldsymbol{\ell}_{m}^{m_{m}}\big)\circ{f}_{0}=P_{\beta,\mathbf{m}}^{n}\lambda^{\beta}z^{\beta}\boldsymbol{\ell}_{1}^{m_{1}}\cdots\boldsymbol{\ell}_{m}^{m_{m}}\big(1+\mathrm{h.o.t.}(\mathcal{B}_{1})\big).
\]
Here, the notation $\mathrm{h.o.t.}(\mathcal B_1)$ means \emph{higher order terms lying in $\mathcal B_1$}. The statement \eqref{eq:shape Rn} then follows simply by $f_{0}^{-n}\circ(\mathrm{id}+{h_1})\circ{f}_{0}^{n+1}=\big(f_{0}^{-n}\circ(\mathrm{id}+{h_1})\circ{f}_{0}^{n}\big)\circ{f}_{0}$
and Neumann's Lemma. Indeed, it can be seen that the coefficient $R_{\beta,\mathbf m}^{n+1}$ in step $n+1$ can be expressed as a linear polynomial, with coefficients independent of $n$, of finitely many coefficients $P_{\beta,\mathbf n}^n$, $\mathbf n<\mathbf m$, from previous step and of $P_{\beta,\mathbf m}^n$. Here, $h_1$ is as in Lemma~\ref{lem:koenigs}, defined by $\mathrm{id}+h_1=h\circ \varphi^{-1}$. 

Let us prove that: 
\begin{equation}
P_{\beta,\mathbf{m}}^{n+1}=C_{\beta,\mathbf{m}}\big(\{R_{\gamma,\mathbf{n}}^{n+1}:\ (\gamma,\mathbf{n})\in\widetilde{\mathcal{H}},\ (\gamma,\mathbf{n})<(\beta,\mathbf{m})\}\big)+\frac{1}{\lambda}R_{\beta,\mathbf{m}}^{n+1}.\label{eq:stat2}
\end{equation}
Then by \eqref{eq:shape Rn} it follows that, for $(\beta,\mathbf{m})\in\widetilde{\mathcal{H}}$
and $n\in\mathbb{N}$,
\begin{align*}
\begin{split}P_{\beta,\mathbf{m}}^{n+1} & =C_{\beta,\mathbf{m}}\big(\{R_{\gamma,\mathbf{n}}^{n+1}:\ (\gamma,\mathbf{n})\in\widetilde{\mathcal{H}},\ (\gamma,\mathbf{n})<(\beta,\mathbf{m})\}\big)\\
 & \qquad+\frac{1}{\lambda}\cdot B_{\beta,\mathbf{m}}\big(\{P_{\beta,\mathbf{n}}^{n}:\ (\beta,\mathbf{n})\in\widetilde{\mathcal{H}},\ \mathbf{n}<\mathbf{m}\}\big)+\lambda^{\beta-1}P_{\beta,\mathbf{m}}^{n}\\
 & =A_{\beta,\mathbf{m}}(\{P_{\gamma,\mathbf{n}}^{n}:\ (\gamma,\mathbf{n})\in\widetilde{\mathcal{H}},\ (\gamma,\mathbf{n})<(\beta,\mathbf{m})\})+\lambda^{\beta-1}P_{\beta,\mathbf{m}}^{n},
\end{split}
\end{align*}
where $C_{\beta,\mathbf{m}}$ and $A_{\beta,\mathbf{m}}$ are linear
polynomials in the variables $\{P_{\gamma,\mathbf{n}}^{n}:(\gamma,\mathbf{n})\in\widetilde{\mathcal{H}},\ (\gamma,\mathbf{n})<(\beta,\mathbf{m})\}$,
whose coefficients are independent of $n$. This proves the lemma.

In order to prove \eqref{eq:stat2}, let ${k}_{0}:={f}_{0}^{-1}-\frac{1}{\lambda}\mathrm{id}$.
It is easy to see that ${k}_{0}$ contains only monomials which are
of order $1$ (in $z$). Let $h_1$, as before, be defined as $\mathrm{id}+h_1=h\circ\varphi^{-1}$,
and $r$ by:
\begin{equation*}
\big(f_{0}^{\circ(-n)}\circ(\mathrm{id}+h_1)\circ f_{0}^{\circ(n+1)}\big)=f_{0}+r.
\end{equation*}
 Then $\mathrm{ord}_{z}(r)>1$ since $\mathrm{ord}_{z}\left(h_1\right)>1$.
From Taylor's expansion (Proposition \ref{prop:formal Taylor}) we obtain: 
\begin{align}
{f}_{0}^{-1}\circ\big({f}_{0}^{\circ(-n)}\circ(\mathrm{id}+{h_1})\circ{f}_{0}^{\circ(n+1)}\big) & =\mathrm{id}+\sum_{i\geq1}\frac{\left({f}_{0}^{-1}\right)^{(i)}({f}_{0})}{i!}r^{i}\nonumber \\
 & =\mathrm{id}+\left(\frac{1}{\lambda}+{k}_{0}'({f}_{0})\right)r+\sum_{i\geq2}\frac{{k}_{0}^{(i)}({f}_{0})}{i!}r^{i}\nonumber \\
 & =\mathrm{id}+\frac{1}{\lambda}r+\sum_{i\geq1}\frac{{k}_{0}^{(i)}({f}_{0})}{i!}r^{i} . \nonumber 
\end{align}
Clearly, by definition \eqref{eq:Rg} of $R_{\beta,\mathbf m}^{n+1}$,
\begin{equation}
\left[\frac{1}{\lambda}r\right]_{\beta,\boldsymbol{m}}=\frac{1}{\lambda}R_{\beta,\boldsymbol{m}}^{n+1},\label{eq:simila1}
\end{equation}
and $(\beta,\mathbf{m})$ in $\sum_{i\geq1}\frac{{k}_{0}^{(i)}\left({f}_{0}\right)}{i!}r^{i}$
can be realized as follows: 
\begin{align}
(\beta,\boldsymbol{m}) & =(\gamma_{1},\boldsymbol{n}_{1})+\cdots+(\gamma_{i},\boldsymbol{n}_{i})+(1-i,\boldsymbol{v}),\nonumber \\
 & =(\gamma_{1},\boldsymbol{n}_{1})+(\gamma_{2}-1,\boldsymbol{n}_{2})+\cdots+(\gamma_{i}-1,\boldsymbol{n}_{i})+(0,\boldsymbol{v}),\label{eq:simila2}
\end{align}
where $(1-i,\boldsymbol{v})\in\mathrm{Supp}\big({k}_{0}^{(i)}({f}_{0})\big)$
and $(\gamma_{1},\boldsymbol{n}_{1}),\ldots,(\gamma_{i},\boldsymbol{n}_{i})\in\mathrm{Supp}({r})$.
Note that $\gamma_{j}>1,\ j=1,\ldots,i,$ and $\boldsymbol{v}>\mathbf{0}_{m}$.
Note that, in \eqref{eq:simila2}, we can subtract -1 from any $(i-1)$ elements $\gamma_{k}$,
$k=1,\ldots,i$. Therefore, it follows from \eqref{eq:simila2} that
\[
(\gamma_{1},\boldsymbol{n}_{1}),\ldots,(\gamma_{i},\boldsymbol{n}_{i})<(\beta,\boldsymbol{m}).
\]
Now \eqref{eq:stat2} follows from Neumann's Lemma, \eqref{eq:simila1}
and \eqref{eq:simila2}.
\end{proof}

\begin{rem}Section~\ref{subsec:convergence koenigs} shows that, after prenormalizing $f$, we can conclude in either of the following two ways. We can take the limit in the valuation topology of a Picard sequence, as described in Remark~\ref{rem:constr}. Alternatively, we can take the limit in the weak topology of a Koenigs sequence with an initial condition having the same first block as the one in the prenormalization. 
\end{rem}

\section{\label{sec:support normalization}The support of the normalizing
transseries}

In this section, we explain why the support of the normalization $\varphi$
of the hyperbolic transseries $f\in\mathcal{L}_{k}^{H}$ depends only on the support of the transseries $f$. To this end, we have to
check that the proof of case (a) and case (b) of the Main Theorem
goes through in the same way in subspaces of $\mathcal{L}_{k}$ with
restricted supports.\\

We introduce the following notation. For $W\subseteq\mathbb{R}_{\geq0}\times\mathbb{Z}^{k}$,
we define $\mathcal{L}_{k}^{W}\subseteq\mathcal{L}_{k}$, $k\in\mathbb{N}$,
as the set of all transseries from $\mathcal{L}_{k}$ whose support
belongs to $W$: 
\begin{equation*}
\mathcal{L}_{k}^{W}:=\left\{ {f}\in\mathcal{L}_{k}:\,\mathrm{Supp}({f})\subseteq W\right\} .
\end{equation*}
The spaces $\mathcal{L}_{k}^{\ge\beta}$, $\beta\ge0$, defined in
\eqref{eq:L greater equal than beta} are examples of spaces $\mathcal{L}_{k}^{W}$
for particular $W$ of the form $\left\{ \left(\alpha,n_{1},\ldots,n_{k}\right)\in\mathbb{R}_{\ge0}\times\mathbb{Z}^{k}:\alpha\ge\beta\right\} $.
Similarly, for $1\leq m\leq k$ and $V$ a well-ordered subset of
$\mathbf{0}_{m}\times\mathbb{Z}^{k-m+1}$, let
\begin{equation*}
\mathcal{B}_{m}^{V}:=\{{R}\in\mathcal{B}_{m},\ \mathrm{Supp}({R})\subseteq V\}\subseteq\mathcal{B}_{m}.
\end{equation*}
For $W\subseteq\mathbb{R}_{\ge0}\times\mathbb{Z}^{k}$, it is easy
to see that the space $\left(\mathcal{L}_{k}^{W},d_{z}\right)$ is
a closed subspace of $\left(\mathcal{L}_{k},d_{z}\right)$. Since
$\left(\mathcal{L}_{k},d_{z}\right)$ is complete by Proposition~\ref{prop:complete spaces},
then $\left(\mathcal{L}_{k}^{W},d_{z}\right)$ is complete as well.
The same holds, for similar reasons, for the metric spaces $\big(\mathcal{B}_{m}^{V},d_{m}\big)$,
$1\leq m\leq k$.\\

\begin{void}
\textbf{Support of $\varphi$ in case (a) of the Main Theorem}. In
view of understanding the support of $\varphi$ in case (a), we introduce
a well-ordered set $W_{\beta}\subset\mathbb{R}_{\ge0}\times\mathbb{Z}^{k}$
such that the space $\mathcal{L}_{k}^{W_{\beta}}$ is invariant under
the action of the operators $\mathcal{T}_{f}$ and $\mathcal{S}_{f}$
given in \eqref{eq:operators S and T}, in the following way. We first
consider the semigroup $W_{1}$ generated by 
\begin{align}
\mathrm{Supp}(f)\cup\big\{\left(\alpha-1,\mathbf m\right) & :\ (\alpha,\mathbf m)\in\mathrm{Supp}({f}-\lambda\cdot\mathrm{id})\big\}\label{eq:DefW1}\\
\text{} & \cup\left\lbrace (0,1,\ldots,0)_{k+1},\,\ldots,(0,\ldots,0,1)_{k+1}\right\rbrace .\nonumber 
\end{align}
Inductively, we introduce an increasing sequence of semigroups $W_{n}$,
$n\in\mathbb{N}_{\geq 1}$, where $W_{n+1}$ is generated by 
\begin{equation}
W_{n}\cup\left\lbrace \left(\beta_{1},\boldsymbol{m}_{1}\right)+\cdots+\left(\beta_{n+1},\boldsymbol{m}_{n+1}\right)-\left(n,\boldsymbol{0}_{k}\right):\left(\beta_{i},\boldsymbol{m}_{i}\right)\in W_{n},\ \beta_{i}>1\right\rbrace .\label{eq:DefW2}
\end{equation}
These sets are well-ordered by Neumann's Lemma. Finally, let
\begin{align}
W:=\bigcup_{n=0}^{\infty}W_{n}\text{ and }W_{\beta} & :=W\cap\big(\mathbb{R}_{\geq\beta}\times\mathbb{Z}^{k}\big).\label{eq:DefWbeta}
\end{align}
The following proposition states that the set $W_{\beta}$ is well-ordered.
\end{void}

\begin{prop}
\label{prop:well-ordered W} Let $\beta\in\mathbb{R}_{>1}$ and suppose
that $W_{\beta}$ is as in \eqref{eq:DefWbeta}. Then $W_{\beta}$
is a well-ordered set with the property that 
\begin{align}
\left(\beta_{1},\boldsymbol{m}_{1}\right)+\cdots+\left(\beta_{m},\boldsymbol{m}_{m}\right)-(m-1,\boldsymbol{0}_{k})\in W_{\beta},\label{eq:property W beta}
\end{align}
for all $\left(\beta_{1},\boldsymbol{m}_{1}\right),\ldots,\left(\beta_{m},\boldsymbol{m}_{m}\right)\in W_{\beta}$,
$m\geq2$. 
\end{prop}

\begin{proof}
Property \eqref{eq:property W beta} follows directly from \eqref{eq:DefW2}
and the fact that $W$ is the increasing union of the sets $W_{n}$,
$n\in\mathbb{N}.$ 

Hence, we only need to prove that $W_{\beta}$ is well-ordered. Since
$W_{\beta}=(\mathbb{R}_{\geq\beta}\times\mathbb{Z}^{k})\cap W$, it
is sufficient to prove that $W$ is well-ordered. In general, an increasing union
of well-ordered sets may not be well-ordered. So we give a proof based
on the specific properties of the sets $W_{n}$. Let $A$ be a nonempty
subset of $W$, and let us prove that $A$ admits a minimal element.

Set $W_{0}:=\emptyset$ and let $I$ be the set of all $n\in\mathbb{N}$
such that $A\cap\left(W_{n}\setminus W_{n-1}\right)\neq\emptyset$.
Let $\mathbf{w}_{n}:=\min\big(A\cap\left(W_{n}\setminus W_{n-1}\right)\big),\ n\in I$.
Such a minimum exists because the sets $W_{n}$, for $n\in I$, are
well-ordered. We have now constructed a sequence $(\mathbf{w}_{n})_{n}$
of minimal elements of the sets $A\cap(W_{n}\setminus W_{n-1})$,
$n\in I$. Clearly, $\min A=\min\{\mathbf{w}_{n}:n\in I\}.$ Therefore,
it is enough to prove that the family $\{\mathbf{w}_{n}:n\in I\}$
has a smallest element. Note that, by \eqref{eq:DefW2}, $\min(W_{n}\cap(\mathbb{R}_{>1}\times\mathbb{Z}^{k}))=\min(W_{n+1}\cap(\mathbb{R}_{>1}\times\mathbb{Z}^{k})),n\in\mathbb{N}$,
and, therefore, $\min(W\cap(\mathbb{R}_{>1}\times\mathbb{Z}^{k}))=\min(W_{1}\cap(\mathbb{R}_{>1}\times\mathbb{Z}^{k}))$.
Now let 
\begin{equation}
\mathbf{w}:=\min(W_{1}\cap(\mathbb{R}_{>1}\times\mathbb{Z}^{k})).\label{eq:tr}
\end{equation}
Take $m_{0}\in I$. By Archimedes' Axiom and since the first coordinate of $\mathbf{w}$ is strictly greater than $1$,
there exists $n_{0}\geq m_{0}$, such that $n\cdot\mathbf{w}-(n-1,\boldsymbol{0}_{k})>\mathbf{w}_{m_{0}}$,
for all $n>n_{0}$. By \eqref{eq:tr}, $\mathbf{w}_{n}\geq n\cdot\mathbf{w}-(n-1,\boldsymbol{0}_{k})>\mathbf{w}_{m_{0}}$,
for every $n\in I$, $n>n_{0}$. This implies that 
\[
\min A=\min\left\lbrace \mathbf{w}_{n}:n\in I\right\rbrace =\min\big(\left\lbrace \mathbf{w}_{i}:i\in I,i\leq n_{0}\right\rbrace \cup\{\mathbf{w}_{m_{0}}\}\big).
\]
The latter set is finite, therefore, the minimum exists.
\end{proof}
In order to obtain the description of $\mathrm{Supp}\left(\varphi\right)$ we proceed as
follows. It is enough to check that the proof of Lemma \ref{lem:properties Sf Tf }
works the same if, in Lemma \ref{lem:properties Sf Tf }, we replace
$\mathcal{L}_{k}^{\ge\beta}$ by $\mathcal{L}_{k}^{W_{\beta}}$, where
$\beta:=\mathrm{ord}_{z}\left(f-f_{0}\right)$. First, we easily check,
by Proposition~\ref{prop:well-ordered W} and by Taylor's expansion (Proposition \ref{prop:formal Taylor}) of the operators $\mathcal{S}_{f}$ and $\mathcal{T}_{f}$
given in \eqref{eq:operators S and T}, that they
leave the spaces $\mathcal{L}_{k}^{W_{\beta}}$ invariant. Then we
have to prove that $\mathcal{T}_{f}$ is a surjection on $\mathcal{L}_{k}^{W_{\beta}}$.
That is, for a given block $z^{\gamma}M_{\gamma}\in\mathcal{L}_{k}^{W_{\beta}}$,
we need to prove that its preimage by $\mathcal{T}_{f}$ belongs to
$\mathcal{L}_{k}^{W_{\beta}}$ as well.

To this end, define, for a well-ordered subset $V$ of $\left\{ 0\right\} \times\mathbb{Z}^{k}$,
the set
\[
H\left(V\right):=\left\langle V\cup\mathrm{Supp}\left(z^{-1}g_{0}\right)\cup\left\{ \left(0,1,0,\ldots,0\right)_{k+1},\ldots,\left(0,0,\ldots,1\right)_{k+1}\right\} \right\rangle .
\]
$H\left(V\right)$ is also well-ordered, by Neumann's Lemma. It is
easy to see that 
\[
\mathcal{B}_{1}^{H\left(\mathrm{Supp}\left(M_{\gamma}\right)\right)}\subseteq\mathcal{B}_{1}
\]
is invariant under the action of $\mathcal{S}_{1}$, where $\mathcal{S}_{1}$
is as defined in \eqref{eq:S1}.
Note that $z^{-1}g_{0}$ is nothing but $Q$ in the proof of
Lemma \ref{lem:properties Sf Tf }. As the space $\left(\mathcal{B}_{1}^{H\left(\mathrm{Supp}\left(M_{\gamma}\right)\right)},d_{1}\right)$
is complete, it follows from the proof of Lemma \ref{lem:properties Sf Tf }
that $\mathcal{S}_{1}$ has a unique fixed point in $\mathcal{B}_{1}^{H\left(\mathrm{Supp}\left(M_{\gamma}\right)\right)}$.
Hence the preimage $z^{\gamma}H_{\gamma}$ of $z^{\gamma}M_{\gamma}$
by $\mathcal{T}_{f}$ belongs to $\mathcal{L}_{k}^{W_{\beta}}$.      

Therefore we can apply Lemma~\ref{lem:properties Sf Tf } to restricted spaces $\mathcal L_k^{W_\beta}$ instead of $\mathcal L_k^{\geq \beta}$ to conclude that $\varphi-\mathrm{id}\in \mathcal L_k^{W_\beta}$.
\begin{void}
\textbf{Support of $\varphi$ in case (b) of the Main Theorem.}
\end{void} 

Let $\varphi _{0}$ be the \emph{prenormalizing} transformation of transseries $f=f_{0}+\mathrm{h.o.t.}$ (where $f$ and $f_{0}$ are as defined in the Main Theorem), which contains only the leading block. It is easy to see that there exists a well-ordered set $W_0\subseteq \{0\}\times \mathbb Z^k$ such that the support of $z^{-1}(\varphi_{0}-\mathrm{id})$ belongs to $W_0$, which depends only on the leading block of the initial transseries $f$. Indeed, by Remark~\ref{rem:constr}, it is obtained as the limit of a Picard sequence with contraction operator depending only on the leading block of $f$, where initial condition $H_0\in \mathcal B_{\geq 1}^+$ can be chosen arbitrarily.                    
Let now $\beta>1$ be such that $\beta:=\mathrm{ord}_{z}(\varphi_{0}\circ f\circ\varphi_{0}^{-1}-f_0)$. Let us define
\begin{equation}\label{eq:Wbeta}
\widetilde{W}_\beta:=\Big\langle W_0\cup W_\beta\cup\{(0,1,\ldots,0)_{k+1},\ldots,(0,\ldots, 0,1)_{k+1}\}\Big\rangle\cap (\mathbb R_{\geq \beta}\times \mathbb Z^k),
\end{equation}
where $W_{\beta}$ is as defined in \eqref{eq:DefWbeta}, and initial $f$ is used
in definition \eqref{eq:DefW1} of $W_{1}$.
It can be checked by Taylor's expansion (Proposition \ref{prop:formal Taylor}) that $$\mathrm{Supp}(z^{-1}(\varphi_0^{-1}-\mathrm{id}))\in \Big\langle W_0\cup\{(0,1,\ldots,0)_{k+1},\ldots,(0,\ldots, 0,1)_{k+1}\}\Big\rangle.$$ Then, for $f_1:=\varphi_{0}\circ f\circ\varphi_{0}^{-1}$,
\[
f_1-f_0\in \mathcal L_k^{\widetilde{W}_\beta}.
\] It can be checked that the set $\widetilde{W}_\beta$ satisfies property \eqref{eq:property W beta} from Proposition~\ref{prop:well-ordered W} and, by the same reasoning as in case $(a)$, that $\mathcal L_k^{\widetilde{W}_\beta}$ is invariant under $\mathcal S_{f_1}$ and $\mathcal T_{f_1}$. Therefore \emph{normalization}
$\varphi_{1}$ reducing $f_1=\varphi_{0}\circ f\circ\varphi_{0}^{-1}$
to the normal form $f_{0}$ belongs to $\varphi_{1}-\mathrm{id}\in\mathcal{L}_{k}^{{\widetilde{W}_\beta}}$.

Finally, putting
\begin{equation}\label{eq:W0}
\widetilde{W}_0=(1,\mathbf 0_k)+W_0,
\end{equation}
for the composition $\varphi:=\varphi_{1}\circ\varphi_{0}$
it holds that $\varphi-\mathrm{id}\in\mathcal{L}_{k}^{\widetilde{W}_0\cup \widetilde W_\beta}$.
\\

To conclude, we have proven in this section the following proposition: 
\begin{prop}[Support of the normalizing transformation $\varphi$]
Let $f(z)=\lambda z+\mathrm{h.o.t.}\in\mathcal{L}_{k}$, $0< \lambda <1$, and let
$\varphi\in\mathcal{L}_{k}^{0}$ be the normalization of $f$ to its
normal form $f_{0}$ from the Main Theorem. In case $($a$)$ of the Main
Theorem, let $f_{1}:=f$. In case $($b$)$, let $f_{1}\in\mathcal{L}_{k}$
be the transseries obtained from $f$ after prenormalization. Then 
\[
\varphi-\mathrm{id}\in\mathcal{L}_{k}^{\widetilde{W}_0\cup \widetilde{W_{\beta}}},
\]
where $\beta:=\mathrm{ord}_{z}(f_1-f_{0})$, $\beta>1$, $\widetilde{W_{\beta}}$ and 
$\widetilde W_0$ as defined in \eqref{eq:Wbeta} and \eqref{eq:W0}. Here, in case $(a)$ of the Main Theorem
we simply put $W_{0}:=\emptyset$.

In particular, the support of normalization depends only on the support
of initial transseries $f$. 
\end{prop}


\def\cprime{$'$}
\providecommand{\bysame}{\leavevmode\hbox to3em{\hrulefill}\thinspace}
\providecommand{\MR}{\relax\ifhmode\unskip\space\fi MR }
\providecommand{\MRhref}[2]{%
	\href{http://www.ams.org/mathscinet-getitem?mr=#1}{#2}
}
\providecommand{\href}[2]{#2}

\medskip

\emph{Addresses:}\

$^{1}$: University of Split, Faculty of Science, Ru\dj era Bo\v skovi\' ca 33, 21000 Split, Croatia, email: dino.peran@pmfst.hr

$^{2}$: University of Zagreb, Faculty of Science, Department of Mathematics, Bijeni\v cka 30, 10000 Zagreb, Croatia, email: maja.resman@math.hr
 
$^{3}$: Institut de Math\' ematiques de Bourgogne (UMR 5584 CNRS), Universit\' e de Bourgogne, Facult\' e des Sciences Mirande, 9 avenue Alain Savary, BP 47870, 21078 Dijon Cedex, France, email: jean-philippe.rolin@u-bourgogne.fr

$^{4}$: Universit\' e de Paris and Sorbonne Universit\' e, CNRS, IMJ-PRG, F-75006 Paris, France, email: tamara.servi@imj-prg.fr
\end{document}